\documentclass[11pt]{amsart}
\usepackage{amsmath,amsfonts,amssymb,amsthm}
\usepackage{latexsym,bm,graphicx}
\usepackage{mathrsfs}
\usepackage{color}
\usepackage{hyperref}

\title[An almost rigidity theorem and its applications]{An almost rigidity theorem and its applications to noncompact RCD(0,N) spaces with linear volume growth}
\author{Xian-Tao Huang}
\address{School of Mathematics\\  Sun Yat-sen University\\ Guangzhou 510275\\ E-mail address: hxiant@mail2.sysu.edu.cn}

\newtheorem{thm}{Theorem}[section]
\newtheorem{prop}[thm]{Proposition}
\newtheorem{lem}[thm]{Lemma}
\newtheorem{cor}[thm]{Corollary}
\newtheorem{notation}[thm]{Notation}

\newtheorem{claim}[thm]{Claim}

\theoremstyle{definition}
\theoremstyle{remark}

\newtheorem{defn}[thm]{Definition}
\newtheorem{rem}[thm]{Remark}

\numberwithin{equation}{section}

\begin{document}

\maketitle
\begin{abstract} The main results of this paper consists of two parts. Firstly, we obtain an almost rigidity theorem which says that on a $\mathrm{RCD}(0,N)$ space, when a domain between two level sets of a distance function has almost maximal volume compared to that of a cylinder, then this portion is close to a cylinder as a metric space. Secondly, we apply this almost rigidity theorem to study noncompact $\mathrm{RCD}(0,N)$ spaces with linear volume growth. More precisely, we obtain the sublinear growth of diameter of geodesic spheres, and study the non-existence of harmonic functions on such $\mathrm{RCD}(0,N)$ spaces.


\vspace*{5pt}

\end{abstract}
\section{Introduction}  

In the recent years, there are lots of researches on metric measure spaces with synthetic definitions of lower Ricci curvature bound.

Using the theory of optimal transport, Lott, Villani (\cite{LV09}) and Sturm (\cite{St06I} \cite{St06II}) independently introduced a notion of `Ricci bounded from below by $K\in \mathbb{R}$ and dimension bounded above by $N\in[1,\infty]$' for general metric measure spaces, which is called $\textmd{CD}(K,N)$-condition.
Later on, Bacher and Sturm (\cite{BS10}) introduced the $\textmd{CD}^{\ast}(K, N)$ condition; Ambrosio, Gigli and Savar\'{e} (\cite{AGS14}) introduced the notion of $\textmd{RCD}(K,\infty)$ spaces (see also \cite{AGMR15} for the simplified axiomatization).
Then for $N<\infty$, the $\textmd{RCD}(K,N)$ (or $\textmd{RCD}^{*}(K,N)$) are considered by many authors, see e.g. \cite{AMS15} \cite{EKS15} \cite{Gig15} \cite{Gig13} etc.
Recall that a $\textmd{RCD}^{*}(K,N)$ space $(X,d,m)$ is a $\textmd{CD}^{*}(K,N)$ space such that the Sobolev space $W^{1,2}(X)$ is a Hilbert space.
The infinitesimally Hilbertian assumption on $\textmd{RCD}^{*}(K,N)$ spaces rules out Finsler geometry from the class of $\textmd{CD}^{*}(K,N)$-spaces.
In a recent paper \cite{CaMi16}, Cavalletti and Milman prove that the $\textmd{RCD}(K,N)$ condition and $\textmd{RCD}^{*}(K,N)$ condition are equivalent provided $m(X)<\infty$.

The compatibility with the smooth Riemannian case and the stability with respect to measured Gromov-Hausdorff convergence are crucial properties of the above definitions of lower Ricci curvature bound.

Cheeger and Colding (\cite{CC96}, \cite{CC97}, \cite{CC00I}, \cite{CC00II}) initiated the study of the measured Gromov-Hausdorff limit of a sequence of manifolds with a common Ricci curvature lower bound (Ricci-limit space for short).
It is interesting to study whether various of properties of manifolds with Ricci curvature lower bound or Ricci limit spaces hold on metric measure spaces with a synthetic definition of lower Ricci curvature bound.

Recall that Cheeger and Colding obtained the almost splitting theorem and the `almost volume cone implies almost metric cone' theorem in \cite{CC96}.
And these two almost rigidity results will imply the corresponding rigidity results for Ricci-limit spaces.
On the other hand, once we have a rigidity result for limit spaces, then we will obtain a corresponding almost rigidity result via an argument by contradiction.

In \cite{Gig13}, Gigli proved a splitting theorem for $\mathrm{RCD}(0,N)$ spaces, see Theorem 1.4 of \cite{Gig13}.
Recently, De Philippis and Gigli proved the `volume cone implies metric cone' theorem on $\textmd{RCD}$ spaces, see Theorem 4.1 of \cite{GigPhi15}.
Then we obtain the corresponding almost rigidity results on $\mathrm{RCD}$ spaces which are corollaries of these results.
We also remark that in \cite{MN14}, Mondino and Naber proved an almost splitting theorem via excess function.

The `almost volume cone implies almost metric cone' theorem is a special case of an almost rigidity theorem considering domains of the manifold with almost maximal volume compared to a warped product metric measure space (i.e. Theorem 4.85 of \cite{CC96}).
Note that in the statement of Theorem 4.85 of \cite{CC96}, there are some assumptions on the mean curvature.

There is no a reference on the generalization of Theorem 4.85 in \cite{CC96} on $\mathrm{RCD}$ spaces.
On the other hand, a precise description of the relation of the parameters appear in the almost rigidity theorem will be helpful in applications.
In this paper we will try to find a suitable generalization of Theorem 4.85 of \cite{CC96} to $\mathrm{RCD}$ spaces.

De Philippis and Gigli's proof of `volume cone implies metric cone' is powerful to handle other volume rigidity theorem, see Section 4 of \cite{GigPhi15} for discussions.
For example, in \cite{H16}, the author applied the strategy of \cite{GigPhi15} to the `volume cylinder' case.

Thus if we want to use the argument by contradiction to obtain an almost rigidity theorem, we need to investigate which conditions are essentially needed in the argument of \cite{GigPhi15}, and then study the stability of these conditions under the measured Gromov-Hausdorff convergence.

The following is one of our main results, which is an almost rigidity theorem.

\begin{thm}\label{main-1.1}
Suppose $(X,d,m)$ is a $\mathrm{RCD}(0,N)$ space, $E$ is a closed subset of $X$ such that $\mathrm{diam}(\partial E)<\bar{c}$.
Denoted by $\varphi(x)=d(x,E)$.
$0<a<c<b$, $0<\alpha'<\alpha<\frac{b-a}{2}$ are real numbers.
Suppose
\begin{enumerate}
  \item $(X,d,m)$ satisfies $\mathrm{MDADF}$ property on $\varphi^{-1}((a,b))$;
  \item the $\delta$-almost maximal volume condition holds for some $\delta\in(0,1)$, i.e.
  \begin{align}\label{1.11111}
  \frac{m(\varphi^{-1}([a,c]))}{m(\varphi^{-1}([a,b]))}\leq(1+\delta)\frac{c-a}{b-a}.
  \end{align}
\end{enumerate}

Then there exists a length extended metric space $(Z,d_{Z})$ which may has more than one path-connected components and satisfies
\begin{align}
d_{GH}(\varphi^{-1}((a+\alpha,b-\alpha)),Z\times(a+\alpha,b-\alpha))\leq\Psi=\Psi(\delta|a,b,c,\alpha, \alpha',\bar{c},N).
\end{align}
Here $\varphi^{-1}((a+\alpha,b-\alpha))$ is endowed with the arc-length distance $d^{\alpha,\alpha'}$, $Z\times(a+\alpha,b-\alpha)$ is endowed with the product metric.
Furthermore, there is a $\Psi$-Gromov-Hausdorff approximation $\Phi:\varphi^{-1}((a+\alpha,b-\alpha))\rightarrow Z\times(a+\alpha,b-\alpha)$ such that for any $x\in \varphi^{-1}((a+\alpha,b-\alpha))$, it holds
\begin{align}
|\varphi(x)-r(\Phi(x))|<\Psi,
\end{align}
where $r:Z\times(a+\alpha,b-\alpha)\rightarrow(a+\alpha,b-\alpha)$ is the projection to the second factor.
In addition, there exist constants $C_{0},C_{1}>0$ depending only on $N,\bar{c},a,b$ such that $Z$ can be chosen to has at most $C_{0}$ path-connected components and the diameter (with respect to $d_{Z}$) of each component is at most $C_{1}$.
\end{thm}

In Theorem \ref{main-1.1}, similar to Theorem 4.85 of \cite{CC96}, $d^{\alpha,\alpha'}$ means the arc-length distance, see Section \ref{sec-arclength} for definitions.

We call $(Z,d_{Z})$ an extended metric space if the extended distance function $d_{Z}:Z\times Z\rightarrow[0,\infty]$ satisfies  the usual axioms in the definition of distance function.
An extended metric space $(Z,d_{Z})$ can be written as a disjoint union of metric spaces, each of them is called a component and is of the form $\{x\in Z\bigl|d_{Z}(x,z)<\infty\}$ for some $z\in Z$.
$(Z,d_{Z})$ is called an length extended metric space if every component of $Z$ is a length space.

The term $\mathrm{MDADF}$ is short for the term `measure decreasing along distance function'.
The precise definition of $\mathrm{MDADF}$ property is given in Section \ref{sec3.1}.
Here we roughly explain its meaning.
On a metric space $(X,d)$, let $\varphi$ be the distance function to a close set $E$.
It is known that $\varphi$ has close relations to optimal transport.
In particular, when $(X,d,m)$ is a $\mathrm{RCD}$ space, $X\setminus E$ has a subset with full measure, denoted by $\mathcal{T}$, such that $\varphi$ gives a equivalent relation $R$ on $\mathcal{T}$ and each equivalent class is a geodesic ray.
This in some sense gives a gradient flow of $-\varphi$.
For a compact subset $K\subset\varphi^{-1}((a,b))$, let us denote by $\Xi_{[s,t]}(K):=\bigcup_{y\in K}R(y)\cap \varphi^{-1}([s,t])$, where $R(y)$ consists of all points in the same equivalent class with $y$.
Then $\mathrm{MDADF}$ property means $\frac{m(\Xi_{[s,t]}(K))}{t-s}$ is monotonically non-increasing with respect to the $t$ variable.
This means the gradient flow of $-\varphi$ increase the volume element.
Hence the $\mathrm{MDADF}$ property substitutes `level sets of $\varphi$ has nonnegative mean curvature' in the smooth version of the almost rigidity theorem.

We give some equivalent characterizations of $\mathrm{MDADF}$ property and obtain some basic properties of $\mathrm{MDADF}$ property in Section \ref{sec4}.

$\mathrm{MDADF}$ property will play an important role when we prove a rigidity theorem following the strategy in \cite{GigPhi15}.

The basic use of $\mathrm{MDADF}$ property is that, when combined with the condition $$\frac{m(\varphi^{-1}([a,c]))}{m(\varphi^{-1}([a,b]))}\leq \frac{c-a}{b-a},$$
we can prove that, on $\varphi^{-1}([a,b])\subseteq X$ (here $(X,d,m)$ is a $\mathrm{RCD}(0,N)$ space), the gradient flow of $-\varphi$ preserves measure, which is the first step in De Philippis and Gigli's strategy.

Similar to \cite{GigPhi15}, in the second step, we need a Laplacian comparison theorem.
Combining the Laplacian comparison theorem and the fact that the gradient flow of $-\varphi$ preserves measure, we can prove $\Delta\varphi=0$ on $\varphi^{-1}((a,b))$.
Then we can further prove that in this case, the Bochner inequality for $\varphi$ holds as an equality, and the gradient flow of $-\varphi$ preserves the Dirichlet energy.
In Section \ref{sec4.2}, we will prove that $\mathrm{MDADF}$ property will imply the Laplacian comparison theorem we need.

After we have finished the properties in the second step, we apply the fact that maps between $\mathrm{RCD}$ spaces which preserve the Cheeger energy must be isometries (see \cite{Gig13}), to obtain that the gradient flow of $-\varphi$ are local isometries.
And all the other steps in \cite{GigPhi15} can be carried out similarly.

To finish the almost rigidity theorem, a remaining important question is that whether the conditions in Theorem \ref{main-1.1} can be passed to measured Gromov-Hausdorff limit.
We will prove a stability result of $\mathrm{MDADF}$ property in Section \ref{sec4.3}.

Thus $\mathrm{MDADF}$ property is essential to obtain an almost rigidity theorem like Theorem \ref{main-1.1}.

We remark that in the `almost volume cone implies almost metric cone' case, no additional assumption similar to $\mathrm{MDADF}$ property is needed, see \cite{Gig13}.
This is because in the proof of `volume cone implies metric cone', we just consider the distance function to a point, and some role $\mathrm{MDADF}$ property takes in Theorem \ref{main-1.1} is replaced by the so-called $\mathrm{MCP}(K,N)$ condition (see \cite{Oh07} \cite{St06II}).
Both the $\mathrm{MCP}(K,N)$ condition and the suitable Laplacian comparison theorem are implied by the $\mathrm{RCD}^{*}(K,N)$ condition.
In fact, in this case, the Laplacian comparison theorem is implied by the $\mathrm{MCP}(K,N)$ condition and the infinitesimally Hilbertian condition, see Remark 5.16 in \cite{Gig15}.

In this paper, we only consider the case that the model space is a product $\mathrm{RCD}(0,N)$ space.
Note that Cheeger and Colding's almost maximal volume theorem considers the case that the model is a general warped product, and in this case the lower curvature bound may be different at different points.
It will be interesting to give a suitable generalization of Cheeger and Colding's almost maximal volume theorem in the most general form under a synthetic definition of lower Ricci curvature bound.
The proofs in \cite{GigPhi15} and the present paper may be helpful when dealing with this question.

In the second part of this paper, we apply Theorem \ref{main-1.1} to study non-compact $\mathrm{RCD}(0,N)$ spaces with linear volume growth.

Recall that a famous theorem of Calabi (\cite{Cala75}) and Yau (\cite{Y76}) says that a noncompact manifolds with nonnegative Ricci curvature has at least linear volume growth, i.e. $\mathrm{Vol}(B_{p}(r))\geq Cr$ for some positive constant $C$.
In fact, to prove such a theorem we only need a Bishop-Gromov volume comparison property.
Since a Bishop-Gromov volume comparison property holds for most of the synthetic definitions of nonnegative Ricci curvature, Calabi and Yau's theorem also holds on more general metric measure spaces.
From Calabi and Yau's theorem, the linear volume growth case is an extremal case for noncompact manifolds (or metric measure spaces) with nonnegative Ricci curvature, and it deserves detailed research.

Sormani has studied manifolds with nonnegative Ricci curvature and linear volume growth, see \cite{Sor98} \cite{Sor99} \cite{Sor00}.
Some of Sormani's results are generalized to metric measure spaces in \cite{H16}.
In this paper, we continue the study of to noncompact $\textmd{RCD}(0,N)$ spaces with linear volume growth.
We obtain the following theorem:

\begin{thm}\label{main-1.2}
Suppose $(X,d,m)$ is a noncompact $\mathrm{RCD}(0,N)$ space with $m(B_{p}(r))\leq Cr$ for some point $p$ and positive constant $C$.
If $(X,d,m)$ does not split, then
\begin{align}\label{1.2}
\lim_{r\rightarrow\infty}\frac{\mathrm{diam}(\partial B_{p}(r))}{r}=0,
\end{align}
where the diameter of $\partial B_{p}(r)$ is computed with respect to the distance $d$.
\end{thm}

In fact, Theorem \ref{main-1.2} is a corollary of the following theorem:

\begin{thm}\label{main-1.3}
Suppose $(X,d,m)$ is a noncompact $\mathrm{RCD}(0,N)$ space with $m(B_{p}(r))\leq Cr$ for some point $p$ and positive constant $C$, and $b$ is the Busemann function associated to some geodesic ray $\gamma$.
Then the diameter of $b^{-1}(r)$ grows sublinearly, i.e.
\begin{align}\label{1.3}
\lim_{r\rightarrow+\infty}\frac{\mathrm{diam}(b^{-1}(r))}{r}=0.
\end{align}
\end{thm}

Theorem \ref{main-1.3} generalizes Theorem 1 of \cite{Sor99} to non-smooth spaces.
Note that in Theorem 1.2 of \cite{H16}, instead of (\ref{1.2}), the author obtained
\begin{align}\label{1.1}
\limsup_{r\rightarrow+\infty}\frac{\mathrm{diam}(b^{-1}(r))}{r}\leq C_{0}\leq2,
\end{align}
which generalizes Theorem 19 of \cite{Sor98} to non-smooth objects.

Similar to the proof in \cite{Sor99}, the almost rigidity Theorem \ref{main-1.1} is a key ingredient in the proof of Theorem \ref{main-1.3}.
In fact, a Busemann function $b$ has close relation to optimal transport, and up to an $m$-negligible set, $b$ gives a partition of $X$ such that each equivalent class coincides with a so-called Busemann ray.
Hence we have a gradient flow of $b$ in some sense.
When $(X,d,m)$ is a noncompact $\mathrm{RCD}(0,N)$ space, one can prove that along the gradient flow of $b$, the volume element is increasing.
Then the linear volume growth condition implies that along the gradient flow of $b$, the measure will become more and more like that of a cylinder.
By Theorem \ref{main-1.1}, we further know that along the gradient flow of $b$, the length distance in the portion between two level sets of $b$ becomes more and more like that of a cylinder, and then Theorem \ref{main-1.3} can be proved.
See Section \ref{sec6} for details.

There have been extensive researches on harmonic functions with polynomial growth on manifolds with nonnegative Ricci curvature, see \cite{Y76} \cite{LT89} \cite{L97} \cite{CM97a} \cite{CM98b} etc.
In \cite{Sor00}, Sormani studied the existence problem of harmonic functions with polynomial growth on manifolds with nonnegative Ricci curvature and linear volume growth.
We generalize the main result of \cite{Sor00} (i.e. Theorem 1 of \cite{Sor00}) in the following theorem:

\begin{thm}\label{main-1.5}
Suppose $(X,d,m)$ is a noncompact $\mathrm{RCD}(0,N)$ space with $m(B_{p}(r))\leq Cr$ for some point $p$ and positive constant $C$.
If there exists a nonconstant harmonic function $f$ such that $|f(x)|\leq C(1+d(x,p))^{k}$ for some constants $k, C>0$,
then $(X,d,m)$ is isomorphic to the product of the Euclidean line $(\mathbb{R},d_{Eucl},\mathcal{L}^{1})$ and another compact metric measure space $(Z, d_{Z}, m_{Z})$.
Moreover,
\begin{enumerate}
  \item if $N\geq 2$, then $(Z, d_{Z}, m_{Z})$ is a $\mathrm{RCD}(0,N-1)$ space;
  \item if $N\in[1, 2)$, then $Z$ is just a point.
\end{enumerate}
\end{thm}
The proof of Theorem 1 in \cite{Sor00} is mainly based on (\ref{1.3}), the gradient estimate as well as a delicate irritation argument.
Since all the ingredients in Sormani's proof are available in $\mathrm{RCD}$ setting, the same argument can be applied.
In this paper we provide a different proof.
The new proof is based on the following proposition, which is a corollary of Theorem \ref{main-1.2}:
\begin{prop}\label{main-1.4}
Suppose $(X,d,m)$ is a noncompact $\mathrm{RCD}(0,N)$ space with $m(B_{p}(r))\leq Cr$ for some point $p$ and positive constant $C$, then
\begin{enumerate}
  \item if $(X,d,m)$ splits, then its tangent cone at infinity is unique and isomorphic to $(\mathbb{R},d_{Eucl},\mathcal{L}^{1})$.
  \item if $(X,d,m)$ does not split, then each tangent cone at infinity $(X_{\infty},p_{\infty},d_{\infty},m_{\infty})$ is isomorphic to a metric measure space of the form $([0,\infty),d_{Eucl},h\mathcal{L}^{1})$, where $h:[0,\infty)\rightarrow (0,\infty)$ is a locally Lipschitz function.
\end{enumerate}
\end{prop}

Our proof of Theorem \ref{main-1.5} is by an argument by contradiction.
Suppose $(X,d,m)$ does not split and admits a nonconstant harmonic function $f$ with polynomial growth, we choose suitable scales to blow down $X$ so that suitable rescalings of $f$ converge to a nonconstant harmonic function on the the tangent cone at infinity.
But on a metric measure space as in (2) of Proposition \ref{main-1.4}, there is no nonconstant harmonic function. See Section \ref{sec7} for details.

\vspace*{10pt}

\noindent\textbf{Acknowledgments.} The author would like to thank Prof. B.-L. Chen, H.-C. Zhang and X.-P. Zhu for encouragement and helpful discussions. The author is partially supported by NSFC 11701580 and 11521101.

\section{Preliminaries}

Throughout this paper, a metric measure space $(X,d,m)$ always satisfies the following: $(X, d)$ is a complete separable locally compact geodesic space, and $m$ is a nonnegative Radon measure with respect to $d$ and finite on bounded sets, $\textmd{supp}(m)=X$.

A curve $\gamma:[0,T]\rightarrow X$ is called a geodesic provided $d(\gamma_{s},\gamma_{t})=\mathrm{Length}(\gamma|_{[s,t]})$ for every $[s,t]\subset[0,T]$, where $\mathrm{Length}(\gamma)$ means the length of the curve $\gamma$.
Let $\mathrm{Geo}(X)\subset C([0,1],X)$ be the set of all geodesics with domain $[0,1]$.
For $t\in[0,1]$, define the evaluation map $e_{t}:\mathrm{Geo}(X)\rightarrow X$ by $e_{t}(\gamma)=\gamma_{t}$.

A map $\gamma:[0,\infty)\rightarrow X$ is called a geodesic ray if $d(\gamma_{s},\gamma_{t})=|s-t|$ for any $s,t>0$.
A map $\gamma:\mathbb{R}\rightarrow X$ is called a line if $d(\gamma_{s},\gamma_{t})=|s-t|$ holds for any $s,t\in \mathbb{R}$. .

Denote by $\mathcal{B}(X)$ the space of Borel subsets of $X$, $\mathcal{P}(X)$ the space of Borel probability measures on $X$, and $\mathcal{P}_{2}(X)\subset \mathcal{P}(X)$ the space of Borel probability measures $\xi$ satisfying $\int_{X} d^{2}(x,y)\xi(dy)<\infty$ for some (and hence all) $x\in X$.

Given a locally Lipschitz function $f:X\rightarrow \mathbb{R}$, the pointwise Lipschitz constant of $f$ at $x$ is defined to be
$$\mathrm{lip}(f)(x)=\limsup_{y\rightarrow x}\frac{|f(x)-f(y)|}{d(x, y)}$$
if $x$ is not isolated, and $\textmd{lip}(f)(x)=0$ if $x$ is isolated.

The Cheeger energy $\mathrm{Ch}:L^{2}(X)\rightarrow[0,+\infty]$ is defined as
$$\mathrm{Ch}(f):= \inf\biggl\{\liminf_{i}\frac{1}{2}\int_{X}|\mathrm{lip}(f_{i})|^{2}dm\biggl|f_{i}\in L^{2}(X)\cap \mathrm{Lip_{b}}(X), \|f_{i}-f\|_{L^{2}}\rightarrow0\biggr\}.$$

The Sobolev space $W^{1,2}(X)$ is defined as $W^{1,2}(X):= \bigl\{f\in L^{2}(X)\bigl|\mathrm{Ch}(f)< \infty\bigr\}$.
$W^{1,2}(X)$ is equipped with the norm
$$\| f\|^{2}_{W^{1,2}}:=\| f\|^{2}_{L^{2}}+2\mathrm{Ch}(f).$$
It is known that for any $f\in W^{1,2}(X)$, there exists $|Df|\in L^{2}(X)$ such that $2\mathrm{Ch}(f)=\int_{X}|Df|^{2}dm$.
$|Df|$ is called the minimal weak upper gradient of $f$.

For an open set $U\subset X$, we define $W^{1,2}(U)$ to be the space of functions $f:U\rightarrow\mathbb{R}$ locally equal to some function in $W^{1,2}(X)$ and satisfy $f, |Df|\in L^{2}(U)$.
We use $W^{1,2}_{\mathrm{loc}}(U)$ to denote the space of functions $f:U\rightarrow\mathbb{R}$ locally equal to some function in $W^{1,2}(X)$.

We say that $(X, d, m)$ is infinitesimally Hilbertian if $W^{1,2}(X)$ is a Hilbert space.
In this paper, the metric measure space $(X, d, m)$ we considers is always assumed to be infinitesimally Hilbertian.

Let $U\subset X$ be an open set, then for any $f,g\in W^{1,2}_{\mathrm{loc}}(U)$, $\langle D f,D g\rangle:U\rightarrow\mathbb{R}$ is $m$-a.e. defined to be
$$\langle D f,D g\rangle:=\inf_{\epsilon>0}\frac{|D(g+\epsilon f)|^{2}-|Dg|^{2}}{2\epsilon},$$
where the infimum is in $m$-essential sense.
The map $W^{1,2}_{\mathrm{loc}}(U)\ni f, g\mapsto \langle D f,D g\rangle\in L^{1}_{\mathrm{loc}}(U)$ is bilinear and symmetric, and we have $\langle D f,D f\rangle=|Df|^{2}$.

Given an open set $U\subset X$, $D(\mathbf{\Delta},U)\subset W^{1,2}_{\mathrm{loc}}(U)$ is the space of $f\in W^{1,2}_{\mathrm{loc}}(U)$ such that there exists a signed Radon measures $\mu$ on $U$ such that
$$\int g d\mu =-\int \langle D f,D g\rangle dm$$
holds for any $g: X\rightarrow\mathbb{R}$ Lipschitz with $\mathrm{supp}(g)\subset\subset U$.
$\mu$ is uniquely determined and we denote it by $\mathbf{\Delta}f$.
If $U=X$, $f\in W^{1,2}(X)\cap D(\mathbf{\Delta},X)$ and $\mathbf{\Delta}f=hm$ for some $h\in L^{2}(X,m)$, then we say $f\in D(\Delta)$, and denote by $\Delta f=h$.

A function $f\in W^{1,2}_{\mathrm{loc}}(U)$ is called harmonic on $U$ if $u\in D(\mathbf{\Delta},U)$ and $\mathbf{\Delta} u=0$ in $U$.

Let's recall some facts on optimal transport. Let $(X, d)$ be a geodesic space, $c:X\times X\rightarrow\mathbb{R}$ be the function $c(x,y)=\frac{d^{2}(x,y)}{2}$.
For $\mu,\nu\in{\mathcal{P}_{2}(X)}$, consider their Wasserstein distance $W_{2}(\mu,\nu)$ defined by
\begin{align}\label{2.4}
W_{2}^{2}(\mu,\nu)=\underset{\pi\in\Gamma(\mu,\nu)}{\min}\int_{X\times{X}}d^{2}(x,y)d\pi(x,y),
\end{align}
where $\Gamma(\mu,\nu)$ is the set of Borel probability measures $\pi$ on $X\times{X}$ satisfying $\pi(A\times{X})=\mu(A)$, $\pi(X\times{A})=\nu(A)$ for every Borel set $A\subset{X}$.
We call a plan $\pi$ that minimizes (\ref{2.4}) an optimal transportation.
$W_{2}$ is a distance on $\mathcal{P}_{2}(X)$, and in fact $(\mathcal{P}_{2}(X),W_{2})$ is a geodesic space.
Furthermore, $W_{2}$ can be equivalently characterized as:
\begin{align}\label{2.3}
W_{2}^{2}(\mu,\nu)={\min}\int\int_{0}^{1}|\dot{\gamma}_{t}|^{2}dtd\Pi(\gamma),
\end{align}
where the minimum is taken among all $\Pi\in \mathcal{P}(C([0, 1],X))$ such that $(e_{0})_{\#}\Pi=\mu$ and $(e_{1})_{\#}\Pi=\nu$.
The set of optimal dynamical plans realizing the minimum in (\ref{2.3}) is denoted by $\textmd{OptGeo}(\mu,\nu)$.

Given a map $\varphi : X\rightarrow \mathbb{R}\cup \{-\infty\}$, its $c$-transform $\varphi^{c} : X\rightarrow \mathbb{R}\cup \{-\infty\}$ is defined to be
$$\varphi^{c}(y)=\inf_{x\in X}\bigl[\frac{d^{2}(x,y)}{2}-\varphi(x)\bigr].$$

Note that
\begin{align}\label{2.2}
\varphi^{cc}\geq \varphi
\end{align}
holds for any $\varphi : X\rightarrow \mathbb{R}\cup \{-\infty\}$.

A function $\varphi : X\rightarrow \mathbb{R}\cup \{-\infty\}$ is called $c$-concave if $\varphi^{cc}= \varphi$.

For a $c$-concave function $\varphi$, let
$$\partial^{c}\varphi:=\{(x,y)\in X\times X| \varphi(x) + \varphi^{c}(y)=\frac{d^{2}(x, y)}{2}\}$$
be its $c$-superdifferential.
Denote by $\partial^{c}\varphi(x):=\{y\in X\bigl|(x, y)\in \partial^{c}\varphi\}$.

It is well known that optimal transportations and $c$-concave functions are closely related to each other, see e.g. Theorem 5.10 in \cite{Vi09} or Theorem 2.13 in \cite{AG11}.

Given a metric measure space $(X,d,m)$ and a number $N\in[1,\infty)$, we define the $N$-R\'{e}nyi entropy functional
$\mathcal{S}_{N}(\cdot|m):\mathcal{P}_{2}(X)\rightarrow[-\infty,0]$ by
$$\mathcal{S}_{N}(\mu|m):=-\int\rho^{-\frac{1}{N}}d\mu,$$
where $\mu=\rho m+\mu^{s}$, $\mu^{s}\bot m$.

Note that for the case $N=1$, $\mathcal{S}_{1}(\mu|m)=-m(\{\rho>0\})$ for $\mu=\rho m+\mu^{s}$, $\mu^{s}\bot m$.

\begin{defn}
Let $N\geq1$.
A metric measure space $(X,d,m)$ is called an $\textmd{RCD}(0, N)$ space if it is infinitesimally Hilbertian and satisfies the following condition (so-called $CD(0,N)$ condition):
for any $\mu_{0},\mu_{1}\in\mathcal{P}_{2}(X)$ with bounded support and $\mu_{0}=\rho_{0}m$, $\mu_{1}=\rho_{1}m$, there exists a geodesic $\{\mu_{t}\}_{t\in[0,1]}\subset\mathcal{P}_{2}(X)$ connecting $\mu_{0},\mu_{1}$ such that
\begin{align}\label{2.1}
\mathcal{S}_{N'}(\mu_{t}|m)\leq (1-t)\mathcal{S}_{N'}(\mu_{0}|m)+t\mathcal{S}_{N'}(\mu_{1}|m)
\end{align}
holds for all $t\in[0,1]$ and $N'\in[N,\infty)$.
\end{defn}

\section{Properties for distance functions}\label{sec3}

In this section, we will always assume that $(X,d,m)$ is a $\mathrm{RCD}(0,N)$ space with $N>1$, $E$ is a closed subset of $X$ such that $\mathrm{diam}(\partial E)\leq\bar{c}$.
Throughout this section, we use $\varphi$ to denote the distance function to $E$, i.e.
$$\varphi(x)=d(x,E).$$

In this section, we fix some notations and recall some properties of distance functions.
We note that in this section, the assumption that $(X,d,m)$ is a $\mathrm{RCD}(0,N)$ space is only for simplicity, and all the (similar) results hold on general $\mathrm{RCD}^{*}(K,N)$ spaces.
\subsection{Properties of distance functions}\label{sec3.1}

Distance functions have close connection to the theory of optimal transport.
Here we consider the cost function $c(x,y)=\frac{d^{2}}{2}(x,y)$.
Let
$\psi=\frac{\varphi^{2}}{2}$,
then
\begin{lem}\label{lem3.01}
$\psi$ is a $c$-concave function.
Furthermore, if $x\in E$, then $\psi^{c}(x)=0$; for any $x\in X$, and any $p\in E$ satisfying $\varphi(x)=d(x,p)$, it holds $(x,p)\in \partial^{c}\psi$.
\end{lem}

\begin{proof}
For any $y\in E$, it is easy to see
$\psi^{c}(y)=\inf_{x\in X}\bigl[\frac{d^{2}(x,y)}{2}-\psi(x)\bigr]=0$, and $(y,y)\in \partial^{c}\psi$.
For any $x\notin E$, let $p\in \partial E$ be a point satisfying $d(x,p)=d(x,E)$, then we have
\begin{align}\label{3.11}
&\psi^{cc}(x)=\inf_{y\in X}\biggl[\frac{d^{2}(x,y)}{2}-\psi^{c}(y)\biggr]\\
\leq &\frac{d^{2}(x,p)}{2}-\psi^{c}(p)
=\psi(x).\nonumber
\end{align}
On the other hand, by (\ref{2.2}), $\psi^{cc}\geq \psi$ always holds, thus $\psi^{cc}(x)=\psi(x)$, and all the inequalities in (\ref{3.11}) must be equalities.
In particular, we have
$$\psi(x)+\psi^{c}(p)=\frac{d^{2}(x,p)}{2},$$
i.e. $(x,p)\in \partial^{c}\psi$.
\end{proof}

Recall the following result from \cite{GRS13}:

\begin{thm}[Theorem 1.3 in \cite{GRS13}]\label{thm3.25}
Suppose $(X,d,m)$ is an $\mathrm{RCD}^{*}(K,N)$ space and $\psi: X\rightarrow \mathbb{R}$ is a $c$-concave function.
Then for $m$-a.e. $x\in X$ there exists exactly one $\eta\in \mathrm{Geo}(X)$ such that $\eta_{0}=x$ and $\eta_{1}\in \partial^{c}\psi(x)$.
\end{thm}

Theorem \ref{thm3.25} is equivalent to the fact that on an $\textmd{RCD}^{*}(K,N)$ space $(X,d,m)$, for every $\mu,\nu\in \mathcal{P}_{2}(X)$ with $\mu\ll m$, there exists a unique plan $\Pi\in\textmd{OptGeo}(\mu,\nu)$ and this $\Pi$ is induced by a map and concentrated on a set of non-branching geodesics.
See \cite{GRS13}.

Denote by
\begin{align}\label{def-T}
\mathcal{T}=\{x\in X\setminus E&\bigl| \text{there exist exactly one point } p\in E \text{ and one geodesic } \\
&\gamma^{x}:[0,1]\rightarrow X \text{ such that }d(x,p)=d(x,E), \gamma^{x}_{0}=x, \gamma^{x}_{1}=p\}.\nonumber
\end{align}

By Lemma \ref{lem3.01} and Theorem \ref{thm3.25}, it is easy to see
\begin{lem}\label{lem3.03}
$m(X\setminus (\mathcal{T}\cup E))=0$.
Furthermore, for any $x\in\mathcal{T}$ and $t\in[0,1)$, $\gamma^{x}_{t}\in \mathcal{T}$, where $\gamma^{x}:[0,1]\rightarrow X$ is the unique geodesic in (\ref{def-T}).
\end{lem}

Denote by
\begin{align}
R&:=\{(x, y)\in X\times X\bigl| |\varphi(y)-\varphi(x)|=d(x,y)\},\nonumber\\
R(x)&:=\{y\in X\bigl| (x,y)\in R\}.\nonumber
\end{align}

It is not hard to prove the following two propositions.

\begin{prop}\label{prop3.04}
For $x\notin E$, let $p\in E$ such that $d(x,p)=\varphi(x)$, and $\gamma$ be a geodesic such that $\gamma_{0}=x$, $\gamma_{1}=p$, then $(\gamma_{s},\gamma_{s})\in R$ for any $s, t\in[0,1]$.
\end{prop}

\begin{prop}\label{prop3.05}
$R$ is an equivalent relation on $\mathcal{T}$,
and for all $x\in\mathcal{T}$, $R(x)\cap \mathcal{T}$ consists of a unit speed geodesic of the form $\gamma:I\rightarrow X\setminus E$ with $I=(0,L)$ or $I=(0,L]$ and satisfies $\varphi(\gamma_{s})=s$ for every $s\in I$.
Here $L\in \mathbb{R}^{+}\cup \{+\infty\}$.
\end{prop}

Then we can choose a subset $Q\subset \mathcal{T}$ such that each equivalent class of $\mathcal{T}$ (with respect to the equivalence relation $R$) has exactly a representative in $Q$.
We call $Q$ a cross-section of $R$.
In fact, $Q$ can be chosen to be locally a level set of $\varphi$.

Let $\mathcal{A}(X)$ denote the $\sigma$-algebra generated by all analytic subsets in $X$.
For every $i\in Z^{+}$, denote by $A_{i}=P_{1}\{(x,y)\in\mathcal{T}\times \mathcal{T}\cap R\bigl|\varphi(y)\geq\frac{1}{i}\}$, where $P_{1}$ is the projection map to the first factor.
Then every $A_{i}$ is analytic, and $A_{i}\subset A_{j}$ if $i\leq j$.
Let $B_{1}:=A_{1}$, $B_{i}:=A_{i}\setminus A_{i-1}$ for $i\geq 2$, and $Q_{i}:=B_{i}\cap \varphi^{-1}(\frac{1}{i})$, then $B_{i}, Q_{i}\in \mathcal{A}(X)$, and $Q:=\bigcup_{i=1}^{\infty}Q_{i}$ is a cross-section.
Define a map $\mathcal{Q}:\mathcal{T}\rightarrow \mathcal{T}$ to be $\mathcal{Q}(x)=R(x)\cap\varphi^{-1}(\frac{1}{i})$ if $x\in B_{i}$, then it is easy to check that $\mathcal{Q}$ is $\mathcal{A}(X)$-measurable.

In conclusion, we have

\begin{lem}\label{f-fcn}
It is possible to construct a $\mathcal{A}(X)$-measurable quotient map $\mathcal{Q}:\mathcal{T}\rightarrow Q$ such that the cross-section $Q=\bigcup_{i=1}^{\infty}Q_{i}$, $Q_{i}\in \mathcal{A}(X)$, $Q_{i}\subset\varphi^{-1}(i^{-1})$.
\end{lem}

\begin{notation}
Let $K\subset \varphi^{-1}((0,\infty))$ be a compact set.
Denote by
\begin{align}\label{symb-1}
\Xi(K):=\bigcup_{y\in K}R(y),
\end{align}
\begin{align}\label{symb-2}
\Xi_{[s,t]}(K):=\Xi(K)\cap\varphi^{-1}([s,t]),
\end{align}
\begin{align}\label{symb-4}
Q(K):=\Xi(K)\cap Q.
\end{align}
\end{notation}

\begin{defn}\label{h0condition}
For any $0< a<b$, we say $(X,d,m)$ satisfies measure-decreasing-along-distance-function ($\mathrm{MDADF}$ for short) property on $\varphi^{-1}((a,b))$ if for any compact set $K\subset\subset \varphi^{-1}((a,b))$ and any $a<r_{1}<r_{2}<r_{3}<\min_{x\in K}\varphi(x)$, it holds
\begin{align}\label{3.17-4}
\frac{m(\Xi_{[r_{1},r_{2}]}(K))}{r_{2}-r_{1}}\geq\frac{m(\Xi_{[r_{2},r_{3}]}(K))}{r_{3}-r_{2}}.
\end{align}
\end{defn}

\begin{rem}
It is easy to see that (\ref{3.17-4}) is equivalent to
\begin{align}\label{3.17-3}
\frac{m(\Xi_{[r_{1},r_{2}]}(K))}{r_{2}-r_{1}}\geq\frac{m(\Xi_{[r_{1},r_{3}]}(K))}{r_{3}-r_{1}}.
\end{align}
\end{rem}

We will give some quivalent characterizations of $\mathrm{MDADF}$ property in Theorem \ref{thm4.1}.
Before this, let's fix some other notations and recall some useful theorem.

\begin{defn}\label{def3.13}
\begin{enumerate}
  \item Define the ray map $g: \mathrm{Dom}(g)\subset Q\times\mathbb{R}^{+}\rightarrow \mathcal{T}$ such that $y=g(q,t)$ is the unique point in $\mathcal{T}$ such that $(y,q)\in R$, $\varphi(y)=t$.
  \item For $t\in \mathbb{R}^{+}$, define a map $F_{t}: \mathcal{T}\cap \varphi^{-1}((t,\infty)) \rightarrow\mathcal{T}$ by
  $$F_{t}(x):=g(y,s-t)$$
  for any $x=g(y,s)\in\mathcal{T}\cap \varphi^{-1}((t,\infty))$.
  \item For any $a>0$, define $\psi_{a}:X\rightarrow\mathbb{R}$ by
\begin{align}\label{3.3}
\psi_{a}(x)=\frac{(\max\{\varphi(x)-a,0\})^{2}}{2}.
\end{align}
\end{enumerate}
\end{defn}

It is easy to see that the function $\max\{\varphi(x)-a,0\}$ is the distance function to the set $\{x\bigl|d(x,E)\leq a\}$.
Hence by Lemma \ref{lem3.01}, $\psi_{a}$ is a $c$-concave function, and for any $x\in\mathcal{T}\cap \varphi^{-1}([a,\infty))$, $\partial^{c}\psi_{a}(x)$ consists of exactly one point $y$ and it satisfies $(x,y)\in R$ and $\varphi(y)=a$.

\begin{defn}
For any $a>0$, define a transport map $F_{1}^{a}:\mathcal{T}\cap \varphi^{-1}([a,\infty))\rightarrow \mathcal{T}$ by $F_{1}^{a}(x)=\partial^{c}\psi_{a}(x)$.
For any $t\in[0,1)$, define a transport map $F_{t}^{a}:\mathcal{T}\cap \varphi^{-1}([a,\infty))\rightarrow \mathcal{T}$ such that $F_{t}^{a}(x)$ is the unique point lying on the unique geodesic connecting $x$ and $F_{1}^{a}(x)$ such that $d(x,F_{t}^{a}(x))=td(x,F_{1}^{a}(x))$,
in other words,
$$F_{t}^{a}(x)=g(q,(1-t)s+ta)$$
for any $x=g(q,s)\in\mathcal{T}\cap \varphi^{-1}([a,\infty))$.
\end{defn}

\begin{rem}\label{rem3.12}
Let $a>0$.
For any $\mu_{0}\in \mathcal{P}_{2}(X)$ with $\mu_{0}\ll m$ and $\mathrm{supp}(\mu_{0})\subset \varphi^{-1}([a,\infty))$,
$\mu_{t}=(F_{t}^{a})_{\#}\mu_{0}$ is well-defined for every $t\in[0,1]$.
Furthermore, by Lemma \ref{lem3.01}, $\pi:=(\mathrm{Id},F_{1}^{a})_{\#}\mu_{0}$ is an optimal transportation and $\psi_{a}$ is a Kantorovich potential (with respect to the cost function $c(x,y)=\frac{d^{2}(x,y)}{2}$) between $\mu_{0}$ and $\mu_{1}$.
Then by Theorem 1.1 in \cite{GRS13}, $\pi$ is the unique optimal transportation, the curve $[0,1]\ni t\mapsto \mu_{t}$ is the unique $L^{2}$-Wasserstein geodesic connecting $\mu_{0}$ and $\mu_{1}$.
In this case, there exists a unique $\Pi\in\mathrm{OptGeo}(\mu_{0},\mu_{1})$, and it holds that $(e_{t})_{\#}\Pi=\mu_{t}$ for every $t\in[0,1]$.
Furthermore, $\mu_{t}\ll m$ for every $t\in[0,1)$.
\end{rem}

\begin{rem}\label{rem3.15}
By the quotient map $\mathcal{Q}$ constructed in Lemma \ref{f-fcn}, we endow $Q$ with the push forward $\sigma$-algebra $\mathfrak{Q}$:
$$C\in \mathfrak{Q}\Leftrightarrow \mathcal{Q}^{-1}(C)\in \mathcal{A}(\mathcal{T}).$$

For any fixed $l>1$, denoted by $\mathcal{T}^{l}=\mathcal{T}\cap\varphi^{-1}((0,l))$.
We endow $(Q,\mathfrak{Q})$ with a measure $\mathfrak{q}$ given by
$\mathfrak{q}:=\mathcal{Q}_{\#}m\llcorner_{\mathcal{T}^{l}}$.

Note that the measure $\mathfrak{q}$ depend on $l$.
\end{rem}

Now we apply the Disintegration Theorem (see Appendix A of \cite{BC09} for a proof) to decompose the measure $m\llcorner_{\mathcal{T}^{l}}$ according to the quotient map $\mathcal{Q}$:

\begin{thm}\label{disin-1}
There is a map $\rho:\mathcal{A}(\mathcal{T})\times Q\rightarrow[0, \infty]$ such that
\begin{enumerate}
  \item $\rho_{q}(\cdot):=\rho(\cdot,q)$ is a probability measure on $(\mathcal{T},\mathcal{A}(\mathcal{T}))$ for every $q\in Q$,
  \item $\rho_{\cdot}(B)$ is $\mathfrak{q}$-measurable for all $B\in\mathcal{A}(\mathcal{T})$,
  \item for all $B\in\mathcal{A}(\mathcal{T})$ and $C\in\mathfrak{Q}$, the following consistency condition holds:
  \begin{align}\label{disintegration}
  m(B\cap \mathcal{Q}^{-1}(C)\cap\varphi^{-1}((0,l)))=\int_{C}\rho_{q}(B)d\mathfrak{q}(q),
  \end{align}
  \item for $\mathfrak{q}$-a.e. $q\in Q$, $\rho_{q}$ is concentrated on $\mathcal{Q}^{-1}(q)\cap\varphi^{-1}((0,l))$.
\end{enumerate}
\end{thm}

We call the map $Q\ni q\mapsto\rho_{q}\in\mathcal{P}(\mathcal{T})$ satisfying (2)-(4) in Theorem \ref{disin-1} a disintegration of $m\llcorner_{\mathcal{T}^{l}}$ strongly consistent with $\mathcal{Q}$.
The measures $\rho_{q}$ are called conditional probabilities.

Since $(X,d,m)$ is a $\textmd{RCD}(0,N)$ space, we have more information on the regularity of the conditional probabilities $\rho_{q}$, see \cite{BC09} \cite{Ca14-2} \cite{CaMo15-1} etc.

\begin{thm}\label{thm3.15}
For $\mathfrak{q}$-a.e. $q\in Q$, $\rho_{q}$ is absolutely continuous with respect to  $g(q,\cdot)_{\#}\mathcal{L}^{1}$.
More precisely, there is some function $h(\cdot,\cdot): \mathrm{Dom}(g)\cap Q\times(0,l)\rightarrow[0,\infty)$
such that
\begin{align}\label{3.5}
m\llcorner_{\mathcal{T}^{l}}=g_{\#}(h\mathfrak{q}\otimes \mathcal{L}^{1})
\end{align}
and
\begin{align}
\rho_{q}=g(q,\cdot)_{\#}(h(q,\cdot)\mathcal{L}^{1})
\end{align}
for $\mathfrak{q}$-a.e. $q\in Q$.
Furthermore, the function $h$ satisfies the following properties:
\begin{description}
  \item[(A)] for $\mathfrak{q}$-a.e. $q\in Q$ and $\sigma_{-}< s\leq t <\sigma_{+}$ such that $(\sigma_{-},\sigma_{+})\subset \mathrm{Dom}(g(q,\cdot))\cap(0,l)$, we have
      \begin{align}\label{3.14}
      \biggl(\frac{\sigma_{+}-t}{\sigma_{+}-s}\biggr)^{N-1}
      \leq \frac{h(q, t)}{h(q, s)}
      \leq \biggl(\frac{t-\sigma_{-}}{s-\sigma_{-}}\biggr)^{N-1};
      \end{align}
  \item[(B)] for $\mathfrak{q}$-a.e. $q\in Q$, and all $s\in[0,1]$, all $t_{0}, t_{1}\in\mathrm{Dom}(g(q,\cdot))\cap(0,l)$ with $t_{0}<t_{1}$, it holds
      \begin{align}\label{3.15}
      h(q,(1-s)t_{0}+st_{1})^{\frac{1}{N-1}}\geq (1-s)h(q,t_{0})^{\frac{1}{N-1}}
      +sh(q,t_{1})^{\frac{1}{N-1}}.
      \end{align}
\end{description}
\end{thm}

Note that \textbf{(A)} in Theorem \ref{thm3.15} implies that for $\mathfrak{q}$-a.e. $q\in Q$, the function $t\mapsto h(q,t)$ is locally Lipschitz.

We remark that in Theorems \ref{disin-1} and \ref{thm3.15}, $\rho_{q}$ and $h(q,\cdot)$ depends on the given $l$. But the monotonicity property and convexity property in the form of (\ref{3.14}) and (\ref{3.15}) are independent of the choice of $l$.

By Theorem \ref{thm3.15} and the arbitrariness of $l$, it is not hard to obtain

\begin{cor}\label{cor3.16}
For any $a>0$, we have $m(\varphi^{-1}(a))=0$.
\end{cor}

\begin{rem}
Theorem \ref{thm3.15} is a special case of results obtained in \cite{BC09} \cite{Ca14-2} \cite{CaMo15-1} etc.
In fact, suppose $(X,d,m)$ is a general $\mathrm{RCD}^{*}(K,N)$ space, given a $1$-Lipschitz function $\varphi$, we have an equivalence relation on the so-called transport set $\mathcal{T}$, so that each equivalence class is a geodesic ray.
Then by the disintegration theorem, we can decompose the measure $m$ according to this equivalence relation.
Making use of the curvature assumptions and the disintegration formula, we can prove that the conditional probabilities satisfy good regularity similar to Theorem \ref{thm3.15}.
This is the so-called $L^{1}$-optimal transportation theory, and it gives many important applications recently.
For example, it plays an important role in Cavalletti and Mondino's proof of L\'{e}vy-Gromov isoperimetric inequality in non-smooth setting (see \cite{CaMo15-1}).
The readers can also refer to \cite{Ca17} for a comprehensive introduction of the theory.
\end{rem}

\subsection{Arc-length distance}\label{sec-arclength}

\begin{defn}\label{def3.17}
For $0< a<b$, $0\leq\alpha'\leq\alpha<\frac{b-a}{2}$, we define
$$d^{\alpha'}(x,y)=\inf\{\mathrm{Length}(\gamma)\bigl| \gamma:[0,1]\rightarrow \varphi^{-1}(a+\alpha', b-\alpha'), \gamma_{0} = x, \gamma_{1} = y\}$$
as a distance function on $\varphi^{-1}(a+\alpha', b-\alpha')$, and $d^{\alpha,\alpha'}$ its restriction on $\varphi^{-1}(a+\alpha, b-\alpha)$.
\end{defn}
From this definition, $x$ and $y$ are not in the same path connected component of $\varphi^{-1}(a+\alpha', b-\alpha')$ if and only if  $d^{\alpha'}(x,y)=\infty$.
In general, the metric space $(\varphi^{-1}(a+\alpha', b-\alpha'),d^{\alpha'})$ is not a geodesic space, because a path of minimal length may have a piece contained in $\varphi^{-1}(a+\alpha')$ or $\varphi^{-1}(b-\alpha')$.

The following lemma is basic but useful.
\begin{lem}\label{lem3.18}
Suppose $x\in \varphi^{-1}(a+\alpha', b-\alpha')$, and $0<r_{0}<\frac{1}{2}\min\{\varphi(x)-a-\alpha',b-\alpha'-\varphi(x)\}$.
Then for any $x_{1},x_{2}\in B_{x}(r_{0})$, a geodesic $\gamma:[0,1]\rightarrow X$ connecting $x_{1}$ and $x_{2}$ must be contained completely in $\varphi^{-1}(\varphi(x)-2r_{0}, \varphi(x)+2r_{0})\subset \varphi^{-1}(a+\alpha', b-\alpha')$.
In particular, $d^{\alpha'}(x_{1},x_{2})=d(x_{1},x_{2})$.
\end{lem}

\begin{proof}
For $x_{1},x_{2}\in B_{x}(r_{0})$, we have $|\varphi(x_{i})-\varphi(x)|\leq d(x,x_{i})<r_{0}$.
Suppose there is $z=\gamma_{t}$ such that $\varphi(z)\geq \varphi(x)+2r_{0}$ or $\varphi(z)\leq \varphi(x)-2r_{0}$, then it is easy to see $|\varphi(z)-\varphi(x_{i})|>r_{0}$, and then
$2r_{0}>d(x_{1},x_{2})=d(x_{1},z)+d(z,x_{2})\geq |\varphi(z)-\varphi(x_{1})|+|\varphi(z)-\varphi(x_{2})|>2r_{0}$, which is a contradiction.
\end{proof}

\subsection{Convergence of distance functions}

\begin{defn}
Let $(X,d_{X})$, $(Y,d_{Y})$ be two metric spaces, a map $\Phi:X\rightarrow Y$ is called an $\epsilon$-Gromov-Hausdorff approximation if
\begin{enumerate}
  \item ($\epsilon$-almost distance preserving:) for every $x_{1},x_{2}\in X$, we have $|d_{Y}(\Phi(x_{1}),\Phi(x_{2}))-d_{X}(x_{1},x_{2})|<\epsilon$;
  \item ($\epsilon$-almost onto:) for every $y\in Y$, there is a $x\in X$ such that $d_{Y}(y,\Phi(x))<\epsilon$.
\end{enumerate}
\end{defn}

We also use the notation $\Phi:(X,d_{X})\rightarrow (Y,d_{Y})$ to emphasize the distance structures.

If $\Phi:X\rightarrow Y$ is an $\epsilon$-Gromov-Hausdorff approximation, we define a map $\tilde{\Phi}:Y\rightarrow X$ such that $x=\tilde{\Phi}(y)$ is some point satisfying $d_{Y}(y,\Phi(x))<\epsilon$,
then $\tilde{\Phi}$ is a $4\epsilon$-Gromov-Hausdorff approximation and
$$d_{X}(\tilde{\Phi}(\Phi(x)),x)<4\epsilon, \quad \quad d_{Y}(\Phi(\tilde{\Phi}(y)),y)<\epsilon$$
holds for any $x\in X$, $y\in Y$.
Such a map $\tilde{\Phi}$ is called an $\epsilon$-inverse of $\Phi$.

\begin{defn}
Suppose $\{(X_{i}, d_{i})\}_{i\in\mathbb{N}}$ is a sequence of metric spaces, we say $(X_{i}, d_{i})$ converge to $(X_{\infty}, d_{\infty})$ in the Gromov-Hausdorff topology if there is a sequence of $\epsilon_{i}$-Gromov-Hausdorff approximations $\Phi_{i}:X_{i}\rightarrow X_{\infty}$, where $\epsilon_{i}\rightarrow0$.
This convergence will be denoted by $(X_{i}, d_{i})\xrightarrow{GH}(X_{\infty}, d_{\infty})$.
If $x_{i}\in X_{i}$, $x_{\infty}\in X_{\infty}$, and $\Phi_{i}(x_{i})\rightarrow x_{\infty}$, then we denote by $x_{i}\xrightarrow{GH} x_{\infty}$ for short.
\end{defn}

\begin{defn}\label{def3.21}
Suppose $(X_{i}, p_{i}, d_{i}, m_{i})$ ($i\in\mathbb{N}$) and $(X_{\infty}, p_{\infty}, d_{\infty}, m_{\infty})$ are locally compact pointed metric measure space,
we say $(X_{i}, p_{i}, d_{i}, m_{i})$ converge to $(X_{\infty}, p_{\infty}, d_{\infty}, m_{\infty})$ in the pointed measured Gromov-Hausdorff topology, denoted by $(X_{i}, p_{i}, d_{i}, m_{i})\xrightarrow{pmGH}(X_{\infty}, p_{\infty}, d_{\infty}, m_{\infty})$ for short,
if there is a sequence $\epsilon_{i}\rightarrow0$, and a sequence of measurable maps $\Phi_{i}:B_{p_{i}}(\epsilon_{i}^{-1})\rightarrow B_{p_{\infty}}(\epsilon_{i}^{-1})$, such that
\begin{enumerate}
  \item $\Phi_{i}$ is an $\epsilon_{i}$-Gromov-Hausdorff approximation, and $\Phi_{i}(p_{i})=p_{\infty}$;
  \item $(\Phi_{i})_{\#}m_{i}\rightarrow m_{\infty}$ as $i\rightarrow\infty$, where the converge is dual to the set of compactly supported continuous functions.
\end{enumerate}
If only (1) happens, then we say $(X_{i}, p_{i}, d_{i})$ converge to $(X_{\infty}, p_{\infty}, d_{\infty})$ in the pointed Gromov-Hausdorff topology, and denote by $(X_{i}, p_{i}, d_{i})\xrightarrow{pGH}(X_{\infty}, p_{\infty}, d_{\infty})$ for short.
\end{defn}

\begin{defn}
Suppose $(X_{i}, p_{i}, d_{i})\xrightarrow{pGH}(X_{\infty}, p_{\infty}, d_{\infty})$ with a sequence of $\epsilon_{i}$-Gromov-Hausdorff approximations $\Phi_{i}$, where $\epsilon_{i}\rightarrow0$.
Suppose that $\{f_{i}\}_{i}$ are functions on $X_{i}$ and $f_{\infty}$ is a function on $X_{\infty}$.
Let $K\subset X_{\infty}$ be a compact set.
If for every $\epsilon>0$, there exists $\delta>0$ such that $|f_{i}(x_{i})-f_{\infty}(x_{\infty})|<\epsilon$ holds for every $i\geq\delta^{-1}$, $x_{i}\in X_{i}$, $x_{\infty}\in K$ with $d_{\infty}(\Phi_{i}(x_{i}),x_{\infty})< \delta$, then we say $f_{i}$ converge to $f_{\infty}$ uniformly on $K$.
\end{defn}

The following theorem is a generalization of the classical Arzela-Ascoli Theorem, see Proposition 27.20 in \cite{Vi09} or Proposition 2.12 in \cite{MN14}.
\begin{thm}\label{AA}
Suppose $(X_{i}, p_{i}, d_{i})\xrightarrow{pGH}(X_{\infty}, p_{\infty}, d_{\infty})$, $R\in(0,\infty]$.
Suppose for every $i$, $f_{i}$ is a Lipschitz function defined on $B_{p_{i}}^{(i)}(R)\subset X_{i}$ and $\mathrm{Lip}f_{i}\leq L$ for some constant $L$ on $B_{p_{i}}^{(i)}(R)$, $| f_{i}(p_{i})|\leq C$  for some constant $C$.
Then there exits a subsequence of $f_{i}$, still denoted by $f_{i}$, and a Lipschitz function $f_{\infty}: B^{(\infty)}_{p_{\infty}}(R)\rightarrow\mathbb{R}$, such that $f_{i}$ converge uniformly to $f_{\infty}$ on $\bar{B}^{(\infty)}_{p_{\infty}}(r)$ for every $r<R$.
\end{thm}

In the remaining part of this subsection, we assume $\{(X_{i},p_{i},d_{i},m_{i})\}_{i\in\mathbb{N}\cup\{\infty\}}$ are $\mathrm{RCD}(0,N)$ spaces, and $(X_{i},p_{i},d_{i},m_{i}) \xrightarrow{pmGH} (X_{\infty},p_{\infty},d_{\infty},m_{\infty})$.
Furthermore, assume $E_{i}\subset X_{i}$ is a closed set such that $p_{i}\in \partial E_{i}$ and $\sup_{i}\mathrm{diam}(\partial E_{i})<\infty$.
Let $\varphi_{i}(x)=d_{i}(x,E_{i})$ be the distance function to $E_{i}$.
By Theorem \ref{AA}, up to a subsequence, $\varphi_{i}$ converge to a $1$-Lipschitz function $\varphi_{\infty}:X_{\infty}\rightarrow[0,\infty)$ uniformly on every compact subset of $X_{\infty}$.
Note that $\varphi_{\infty}(p_{\infty})=0$.
Denoted by $E_{\infty}=\{x\in X_{\infty}\bigl|\varphi_{\infty}(x)=0\}$.
We have
\begin{lem}\label{lem3.05}
$\varphi_{\infty}$ coincides with the distance function $d_{\infty}(\cdot,E_{\infty})$.
\end{lem}
\begin{proof}
For any $x_{\infty}\notin E_{\infty}$, suppose $X_{i}\ni x_{i}\xrightarrow{GH}x_{\infty}$, find $y_{i}\in E_{i}$ such that $d_{i}(x_{i},y_{i})=\varphi_{i}(x_{i})$.
Suppose (up to a subsequence) $y_{i}\xrightarrow{GH}y_{\infty}$, then $\varphi_{\infty}(y_{\infty})=\lim_{i\rightarrow\infty} \varphi_{i}(y_{i})=0$
and
$$d_{\infty}(x_{\infty},y_{\infty})=\lim_{i\rightarrow\infty} d_{i}(x_{i},y_{i})= \lim_{i\rightarrow\infty}\varphi_{i}(x_{i})=\varphi_{\infty}(x_{\infty}).$$
Hence $d_{\infty}(x_{\infty},E_{\infty})\leq\varphi_{\infty}(x_{\infty})$.
On the other hand, for any $y'_{\infty}\in E_{\infty}$, $d_{\infty}(x_{\infty},y'_{\infty})\geq \varphi_{\infty}(x_{\infty})-\varphi_{\infty}(y'_{\infty})=\varphi_{\infty}(x_{\infty})$.
Thus $d_{\infty}(x_{\infty},E_{\infty})=\varphi_{\infty}(x_{\infty})$.
\end{proof}

\begin{lem}\label{lem3.06}
For any $0<a<b$, we have
\begin{align}\label{3.3331}
m_{\infty}(\varphi_{\infty}^{-1}([a,b]))=\lim_{i\rightarrow\infty} m_{i}(\varphi_{i}^{-1}([a,b])).
\end{align}
\end{lem}

\begin{proof}
By the uniform convergence of $\varphi_{i}$, for any $\epsilon>0$, $\Phi_{i}(\varphi_{i}^{-1}([a,b]))$ is contained in $\varphi_{\infty}^{-1}([a-\epsilon,b+\epsilon])$ for $i$ sufficiently large.
Hence by the weak convergence of $(\Phi_{i})_{\#}m_{i}$ to $m_{\infty}$ and the arbitrariness of $\epsilon$, it is easy to obtain
$$\limsup_{i\rightarrow\infty}m_{i}(\varphi_{i}^{-1}([a,b]))\leq m_{\infty}(\varphi_{\infty}^{-1}([a,b])).$$

On the other hand, for any open set $A$ such that $A\subset\subset\varphi_{\infty}^{-1}((a,b))$, for $i$ sufficiently large, $\Phi_{i}^{-1}(A)$ is contained in $\varphi_{i}^{-1}((a,b))$, hence by the weak convergence of $(\Phi_{i})_{\#}m_{i}$ to $m_{\infty}$ and arbitrariness of $A$, we have
$$m_{\infty}(\varphi_{\infty}^{-1}((a,b)))\leq\liminf_{i\rightarrow\infty}m_{i}(\varphi_{i}^{-1}((a,b))).$$
Since by Corollary \ref{cor3.16}, $m_{\infty}(\varphi_{\infty}^{-1}(\{a,b\}))=0$ holds, we obtain (\ref{3.3331}).
\end{proof}

\section{Properties under \textmd{MDADF} property}\label{sec4}

\subsection{Equivalent characterization of \textmd{MDADF} property}\label{sec4.1}

\begin{thm}\label{thm4.1}
Suppose $(X,d,m)$ is a $\mathrm{RCD}(0,N)$ space, $E$ is a closed subset such that $\mathrm{diam}(\partial E)<\infty$.
Let $\varphi(x)=d(x,E)$ be the distance function to $E$.
Then for any $0< a<b$, the following statements are equivalent:
\begin{enumerate}
  \item $(X,d,m)$ satisfies $\mathrm{MDADF}$ property on $\varphi^{-1}((a,b))$.
  \item $(X,d,m)$ satisfies $\mathrm{MDADF}$ property on $\varphi^{-1}((a',b'))$ for any $a',b'$ with $a<a'<b'<b$.
  \item If we choose any $l>\max\{b,1\}$ and then define $\mathfrak{q}$ as in Remark \ref{rem3.15} and let $h(q,t)$ be given by Theorem \ref{thm3.15}.
      Then for $\mathfrak{q}$-a.e. $q\in Q$,
      \begin{align}\label{4.1}
      h(q,r_{2})\leq h(q,r_{1})
      \end{align}
      holds for any $r_{1},r_{2}\in \mathrm{Dom}(g(q,\cdot))$ with $a<r_{1}<r_{2}<b$.
  \item For every $\mu_{0}\in \mathcal{P}_{2}(X)$ with $\mu_{0}\ll m$ and $\mathrm{supp}(\mu_{0})\subset \varphi^{-1}([a,b])$,
      suppose $[0,1]\ni t\mapsto \mu_{t}=(F_{t}^{a})_{\#}\mu_{0}$ is the unique $L^{2}$-Wasserstein geodesic defined in Remark \ref{rem3.12}, and
      let $\Pi\in\mathrm{OptGeo}(\mu_{0},\mu_{1})$ be the unique optimal dynamical plan, $\rho_{t}=\frac{d\mu_{t}}{dm}$ be the density function for each $t\in[0,1)$,
      then for every $t\in[0,1)$,
      \begin{align}\label{4.01}
      \rho_{t}(\gamma_{t})\leq \frac{1}{1-t}\rho_{0}(\gamma_{0})
      \end{align}
      holds for $\Pi$-a.e. $\gamma\in\mathrm{Geo}(X)$.
  \item For every $\mu_{0}\in \mathcal{P}_{2}(X)$ as in (4), $[0,1]\ni t\mapsto \mu_{t}=(F_{t}^{a})_{\#}\mu_{0}$ as in Remark \ref{rem3.12}, then
      \begin{align}\label{H-entropy}
      \mathcal{S}_{N'}(\mu_{t}|m)\leq (1-t)\mathcal{S}_{N'}(\mu_{0}|m)
      \end{align}
      holds for every $N'\geq1$.
\end{enumerate}
\end{thm}

\begin{proof}
$(1) \Leftrightarrow(2)$. It is obvious.

$(1)\Rightarrow(3)$.
For any compact set $K\subset\subset \varphi^{-1}((a,b))$,
by (\ref{disintegration}) and Theorem \ref{thm3.15}, for any $\min_{x\in K}\varphi(x)> r_{2}>r_{1}> a$, we have
\begin{align}\label{3.8}
m(\Xi_{[r_{1},r_{2}]}(K))=\int_{Q(K)}\biggl(\int_{r_{1}}^{r_{2}}h(q,s)ds\biggr)d\mathfrak{q}(q).
\end{align}
Thus (\ref{3.17-4}) is equivalent to
\begin{align}\label{3.8-1}
\frac{1}{r_{2}-r_{1}}\int_{Q(K)}\biggl(\int_{r_{1}}^{r_{2}}h(q,s)ds\biggr)d\mathfrak{q}(q)
\geq\frac{1}{r_{3}-r_{2}}\int_{Q(K)}\biggl(\int_{r_{2}}^{r_{3}}h(q,s)ds\biggr)d\mathfrak{q}(q)
\end{align}
holds for any compact set $K\subset\subset \varphi^{-1}((a,b))$ and any $\min_{x\in K}\varphi(x)>r_{3}> r_{2}>r_{1}> a$.

Then by change of variable and the fact that $h(q,\cdot)$ is locally Lipschitz for $\mathfrak{q}$-a.e. $q\in Q$, it is not hard to prove that there is a set $\hat{Q}\subset Q$, such that $\mathfrak{q}(Q\setminus \hat{Q})=0$ and for every $q\in \hat{Q}$, (\ref{4.1}) holds for any $r_{1},r_{2}\in \mathrm{Dom}(g(q,\cdot))$ with $a<r_{1}<r_{2}<b$.

$(3)\Rightarrow(1)$. By integration, we obtain (\ref{3.8-1}), hence obtain (1).

$(3)\Rightarrow(4)$. Fix $t\in[0,1)$. 
For any bounded Borel set $B\subset\varphi^{-1}([a,\infty))$, by (\ref{disintegration}) and Theorem \ref{thm3.15}, we have
\begin{align}\label{4.9-1}
&\mu_{t}(F_{t}^{a}(B))\\
=&\int_{\mathcal{Q}(F_{t}^{a}(B))}\biggl(\int_{g(q,\bar{s})\in F_{t}^{a}(B)}\rho_{t}(g(q,\bar{s}))h(q,\bar{s})d\mathcal{L}^{1}(\bar{s}) \biggr)d\mathfrak{q}(q)\nonumber\\ =&(1-t)\int_{\mathcal{Q}(B)}\biggl(\int_{g(q,s)\in B} \rho_{t}(g(q,(1-t)s+ta))h(q,(1-t)s+ta)d\mathcal{L}^{1}(s)\biggr) d\mathfrak{q}(q).\nonumber
\end{align}

On the other hand,
\begin{align}\label{4.9-2}
&\mu_{t}(F_{t}^{a}(B))=\mu_{0}(B)\\
=&\int_{\mathcal{Q}(B)}\biggl(\int_{g(q,s)\in B}\rho_{0}(g(q,s))h(q,s)d\mathcal{L}^{1}(s) \biggr)d\mathfrak{q}(q).\nonumber
\end{align}
By (\ref{4.1}) and the arbitrariness of $B$,
\begin{align}\label{4.9-3}
\rho_{0}(x)=(1-t)\rho_{t}(F_{t}^{a}(x))\frac{h(q,(1-t)s+ta)}{h(q,s)}\geq(1-t)\rho_{t}(F_{t}^{a}(x))
\end{align}
holds for $\mu_{0}$-a.e. $x=g(q,s)$.

By the uniqueness of the optimal transportation, for $\Pi$-a.e. $\gamma$, $\gamma_{t}=F_{t}^{a}(\gamma_{0})$.
Thus by (\ref{4.9-3}), (\ref{4.01}) holds for $\Pi$-a.e. $\gamma$.

$(4)\Rightarrow(5)$. For any $N'\geq1$ and $t\in[0,1)$,
\begin{align}
&\mathcal{S}_{N'}(\mu_{t}|m)=-\int\rho_{t}(x)^{-\frac{1}{N'}}d\mu_{t}(x)=-\int\rho_{t}(\gamma_{t})^{-\frac{1}{N'}} d\Pi(\gamma)\\
\leq&-\int\biggl(\frac{1}{1-t}\rho_{0}(\gamma_{0})\biggr)^{-\frac{1}{N'}}d\Pi(\gamma)\leq -(1-t)\int\rho_{0}(x)^{-\frac{1}{N'}}d\mu_{0}(x) \nonumber \\
=&(1-t)\mathcal{S}_{N'}(\mu_{0}|m).\nonumber
\end{align}

$(5)\Rightarrow(3)$.
To prove (3), without loss of generality, we may assume there exists $\bar{b}>a$ such that $(0,\bar{b})\subset \mathrm{Dom}(g(q,\cdot))$ for every $q\in Q$, and we will prove (\ref{4.1}) holds for $\mathfrak{q}$-a.e. $q\in Q$ for any $r_{1},r_{2}$ with $a<r_{1}<r_{2}<\min\{b,\bar{b}\}$.

For any $\mathfrak{q}$-measurable subset $C\subset Q$ with $S=\mathfrak{q}(C)>0$, and any $A,L>0$ with  $a<A-L<A<\min\{b,\bar{b}\}$, define $\mu_{0}\in\mathcal{P}_{2}(X)$ to be
$$\mu_{0}:=\frac{1}{S}\int_{C}(g(q,\cdot))_{\#}\biggl(\frac{1}{L}\mathcal{L}^{1}\llcorner_{[A-L,A]}\biggr) d\mathfrak{q}(q).$$
Let $[0,1]\ni t\mapsto \mu_{t}=(F_{t}^{a})_{\#}\mu_{0}$ be the unique $L^{2}$-Wasserstein geodesic defined in Remark \ref{rem3.12}, then it is easy to see that for $t\in[0,1)$,
$$\mu_{t}=\frac{1}{S}\int_{C}(g(q,\cdot))_{\#}\biggl(\frac{1}{L_{t}}\mathcal{L}^{1}\llcorner_{[A_{t}-L_{t},A_{t}]} \biggr) d\mathfrak{q}(q),$$
where $A_{t}=(1-t)A+ta$, $L_{t}=(1-t)L$.
In particular,  for $t\in[0,1)$,
$$\rho_{t}(g(q,s))=\frac{1}{SL_{t}h(q,s)},\quad \forall s\in[A_{t}-L_{t},A_{t}],q\in C,$$
gives the density of $\mu_{t}$ with respect to $m$.
Hence by (\ref{H-entropy}), we have
\begin{align}
&-\int_{C}\biggl(\int_{A-L}^{A}h(q,s)d\mathcal{L}^{1}(s) \biggr)d\mathfrak{q}(q)\nonumber\\
\geq&-\int_{C}\biggl(\int_{A-L}^{A}h(q,(1-t)s+ta)d\mathcal{L}^{1}(s) \biggr)d\mathfrak{q}(q).\nonumber
\end{align}

Thus by the arbitrariness of $C$, $A$, $L$,
\begin{align}\label{4.3}
h(q,s)\leq h(q,(1-t)s+ta)
\end{align}
holds for $\mathfrak{q}$-a.e. $q\in Q$ and $\mathcal{L}^{1}$-a.e. $s$ with $a<s<\min\{b,\bar{b}\}$.

Since $h(q,\cdot)$ is locally Lipschitz for $\mathfrak{q}$-a.e. $q$, it is not hard to prove that (\ref{4.3}) holds for $\mathfrak{q}$-a.e. $q\in Q$, any $s$ with $a<s<\min\{b,\bar{b}\}$ and $t\in[0,1)$.

For any $r_{1},r_{2}$ with $a<r_{1}<r_{2}<\min\{b,\bar{b}\}$, by a suitable choice of $s$ and $t$, we derive that (\ref{4.1}) holds.
\end{proof}

\subsection{Laplacian comparison}\label{sec4.2}

\begin{thm}\label{thm4.3}
Suppose $(X,d,m)$ is a $\mathrm{RCD}(0,N)$ space, $E$ is a closed subset such that $\mathrm{diam}(\partial E)<\infty$.
Let $\varphi(x)=d(x,E)$.
Fix $0<a<b$.
Suppose for every $\mu_{0}\in \mathcal{P}_{2}(X)$ with $\mu_{0}\ll m$ and $\mathrm{supp}(\mu_{0})\subset \varphi^{-1}([a,b])$, (\ref{H-entropy}) holds for some $N'>1$,
where $[0,1]\ni t\mapsto \mu_{t}=(F_{t}^{a})_{\#}\mu_{0}$ is the unique $L^{2}$-Wasserstein geodesic defined in Remark \ref{rem3.12}.
Then $\varphi\in D(\mathbf{\Delta},\varphi^{-1}((a,b)))$, and
\begin{align}\label{Laplace-1}
\mathbf{\Delta}\varphi \leq \frac{N'-1}{\varphi-a}m
\end{align}
on $\varphi^{-1}((a,b))$.
\end{thm}

In \cite{Gig15}, Gigli proved a Laplacian comparison estimates for general locally Lipschitz $c$-concave function on $\mathrm{RCD}$ spaces, see Theorem 5.14 of \cite{Gig15}.
The key in Gigli's proof is to combine a lower bound and an upper bound of the derivative of the entropy functional along a $L^{2}$-Wasserstein geodesic.
The upper bound of derivative of the entropy functional is implied by the curvature assumption.
The lower bound of derivative of the entropy functional is given in Proposition 5.10 of \cite{Gig15}.
Note that there are some technical assumptions in Proposition 5.10 of \cite{Gig15}, and these assumptions are independent of the curvature assumption.

To prove Theorem \ref{thm4.3}, we just follow the strategy in \cite{Gig15}.
In particular, we will check that the geometric assumptions in Theorem \ref{thm4.3} imply those in Proposition 5.10 of \cite{Gig15}, and (\ref{H-entropy}) is enough to obtain the desired Laplacian comparison.

\begin{prop}\label{prop4.1}
Suppose the $\mathrm{RCD}(0,N)$ space $(X,d,m)$, the closed set $E$, distance function $\varphi$, positive numbers $0<a<b$, $N'>1$, the region $\varphi^{-1}([a,b])$ satisfy all the assumptions of Theorem \ref{thm4.3}.
Suppose $\mu\in \mathcal{P}_{2}(X)$ satisfies $\mathrm{supp}(\mu)\subset\varphi^{-1}([a,b])$, $\mu\ll m$ and that $\rho^{1-\frac{1}{N'}}$ is Lipschitz on $\varphi^{-1}([a,b])$, where $\rho=\frac{d\mu}{dm}$ is the density function.
Then
\begin{align}\label{4.6}
-\mathcal{S}_{N'}(\mu|m) \geq-\frac{1}{N'}\int_{\Omega} \langle D(\rho^{1-\frac{1}{N'}}),D\psi_{a}\rangle dm,
\end{align}
where $\psi_{a}$ is the a Kantorovich potential defined in (\ref{3.3}).
\end{prop}

\begin{proof}[Proof of Proposition \ref{prop4.1}]
Let $\Omega$ be the bounded open set $\varphi^{-1}((a,b))$.
Note that by Corollary \ref{cor3.16}, $m(\partial\Omega)=0$.

For every $\epsilon> 0$, define $\rho_{\epsilon}: X\rightarrow\mathbb{R}^{+}$ such that $\rho_{\epsilon}(x)=0$ for $x\in X\setminus \bar{\Omega}$ and $\rho_{\epsilon}(x)=c_{\epsilon}(\epsilon+\rho(x)^{1-\frac{1}{N'}})^{\frac{N'}{N'-1}}$ for $x\in  \bar{\Omega}$, where $c_{\epsilon}$ is a constant such that $\mu_{\epsilon}:=\rho_{\epsilon}m$ is a probability.
Note that $c_{\epsilon}\uparrow1$ as $\epsilon\downarrow0$.
By the definition of $\rho_{\epsilon}$, it is easy to see that
\begin{align}
\int_{\Omega} \langle D(\rho_{\epsilon}^{1-\frac{1}{N'}}),D\psi_{a}\rangle dm=c_{\epsilon}^{1-\frac{1}{N'}}\int_{\Omega} \langle D(\rho^{1-\frac{1}{N'}}),D\psi_{a}\rangle dm,\nonumber
\end{align}
hence
\begin{align}\label{5.6}
\lim_{\epsilon\rightarrow0}\int_{\Omega} \langle D(\rho_{\epsilon}^{1-\frac{1}{N'}}),D\psi_{a}\rangle dm=\int_{\Omega} \langle D(\rho^{1-\frac{1}{N'}}),D\psi_{a}\rangle dm.
\end{align}
It is easy to see that
\begin{align}\label{4.8}
\mathcal{S}_{N'}(\mu_{\epsilon}|m)\rightarrow\mathcal{S}_{N'}(\mu|m)
\end{align}
when $\epsilon\rightarrow0$.

Let $[0,1]\ni t\mapsto \mu_{\epsilon,t}=(F_{t}^{a})_{\#}\mu_{\epsilon}$ be the unique $L^{2}$-Wasserstein geodesic connecting $\mu_{\epsilon,0}=\mu_{\epsilon}$ and $\mu_{\epsilon,1}$ as in Remark \ref{rem3.12}.
In this case, there exists a unique $\Pi_{\epsilon}\in\mathrm{OptGeo}(\mu_{\epsilon,0},\mu_{\epsilon,1})$ such that $(e_{t})_{\#}\Pi_{\epsilon}=\mu_{\epsilon,t}$ for every $t\in[0,1]$.
For any $\epsilon$, $\psi_{a}$ is a locally Lipschitz Kantorovich potential inducing $\Pi_{\epsilon}$.
Note that $\mu_{\epsilon,t}$ is concentrated on $\bar{\Omega}$ for every $t\in[0,1]$.
In addition, by construction, $\rho_{\epsilon}$ is Lipschitz and bounded from below by a positive constant (depending on $\epsilon$) on $\bar{\Omega}$.
Hence by Proposition 5.10 of \cite{Gig15}, we have
\begin{align}\label{4.9}
\liminf_{t\downarrow0}\frac{\mathcal{S}_{N'}(\mu_{\epsilon,t}|m)-\mathcal{S}_{N'}(\mu_{\epsilon,0}|m)}{t} \geq-\frac{1}{N'}\int_{\Omega} \langle D(\rho_{\epsilon}^{1-\frac{1}{N'}}),D\psi_{a}\rangle dm.
\end{align}

On the other hand, by assumptions of Proposition \ref{prop4.1}, $\mu_{\epsilon,t}$ satisfies (\ref{H-entropy}) for $N'$, hence we have
\begin{align}\label{4.10}
-\mathcal{S}_{N'}(\mu_{\epsilon,0}|m)\geq \frac{\mathcal{S}_{N'}(\mu_{\epsilon,t}|m)-\mathcal{S}_{N'}(\mu_{\epsilon,0}|m)}{t}.
\end{align}
By (\ref{5.6}), (\ref{4.8}), (\ref{4.9}) and (\ref{4.10}), we obtain (\ref{4.6}).
\end{proof}

\begin{proof}[Proof of Theorem \ref{thm4.3}]
For any nonnegative Lipschitz function $f:X\rightarrow \mathbb{R}$ such that $\mathrm{supp}f\subset\subset \varphi^{-1}((a,b))=\Omega$ and $f$ is not identically $0$, define $\rho:= cf^{\frac{N'}{N'-1}}$, with $c:=(\int f^{\frac{N'}{N'-1}})^{-1}$ being the normalization constant.
Take $\mu=\rho m$ and apply Proposition \ref{prop4.1} we obtain
\begin{align}
c^{\frac{N'-1}{N'}}\int_{\Omega} f dm=&\int_{\Omega}\rho^{1-\frac{1}{N'}}dm=-\mathcal{S}_{N'}(\mu|m) \\ \geq&-\frac{1}{N'}\int_{\Omega} \langle D(\rho^{1-\frac{1}{N'}}),D\psi_{a}\rangle dm\nonumber\\ =&-\frac{c^{\frac{N'-1}{N'}}}{N'}\int_{\Omega} \langle Df,D\psi_{a}\rangle dm.\nonumber
\end{align}
Hence
\begin{align}
-\int_{\Omega}\langle D\psi_{a},Df\rangle dm\leq N'\int_{\Omega} f dm.
\end{align}
Since $f$ is arbitrary, by Proposition 4.13 of \cite{Gig15}, we have $\psi_{a} \in D(\mathbf{\Delta},\varphi^{-1}((a,b)))$, and
$$\mathbf{\Delta}\psi_{a}\leq N'm.$$
Since $\psi_{a}=\frac{(\varphi-a)^{2}}{2}$ on $\varphi^{-1}((a,b))$, $|D(\varphi-a)|=1$ $m$-a.e. on $\varphi^{-1}((a,b))$, by the chain rule of distributional Laplacian, we have
$$\mathbf{\Delta}\varphi \leq \frac{N'-1}{\varphi-a}m.$$
The proof is completed.
\end{proof}

\begin{cor}\label{cor4.5}
Suppose $(X,d,m)$ is a $\mathrm{RCD}(0,N)$ space, $E$ is a closed subset such that $\mathrm{diam}(\partial E)<\infty$.
Let $\varphi(x)=d(x,E)$ be the distance function to $E$.
Fix $0< a<b$.
Suppose $(X,d,m)$ satisfies $\mathrm{MDADF}$ property on $\varphi^{-1}((a,b))$, then
\begin{align}\label{Laplace-2}
\mathbf{\Delta}\varphi\leq 0
\end{align}
on $\varphi^{-1}((a,b))$.
\end{cor}

\begin{proof}
If $(X,d,m)$ satisfies $\mathrm{MDADF}$ property on $\varphi^{-1}((a,b))$, then by Theorem \ref{thm4.1}, (\ref{H-entropy}) holds for all $N'>1$.
Thus by Theorem \ref{thm4.3},
$$
\mathbf{\Delta}\varphi\leq \frac{N'-1}{\varphi-a}m
$$
holds on $\varphi^{-1}((a,b))$ for all $N'>1$.
Let $N'\downarrow1$, we obtain (\ref{Laplace-2}).
\end{proof}

\subsection{Stability of \textmd{MDADF} property}\label{sec4.3}

\begin{thm}\label{thm4.5}
Suppose $\{(X_{i},p_{i},d_{i},m_{i})\}_{i\in\mathbb{N}}$ is a sequence of $\mathrm{RCD}(0,N)$ spaces, and $(X_{i},p_{i},d_{i},m_{i}) \xrightarrow{pmGH} (X_{\infty},p_{\infty},d_{\infty},m_{\infty})$.
Let $E_{i}\subset X_{i}$ be a closed set such that $p_{i}\in \partial E_{i}$ and $\sup_{i}\mathrm{diam}(\partial E_{i})<\infty$.
Let $\varphi_{i}(x)=d_{i}(x,E_{i})$ be the distance function to $E_{i}$.
Suppose $\varphi_{i}$ converge uniformly on any compact set to  a distance function $\varphi_{\infty}:X_{\infty}\rightarrow[0,\infty)$ (In face, $\varphi_{\infty}(x)=d_{\infty}(x,E_{\infty})$, where $E_{\infty}=\varphi_{\infty}^{-1}(0)$.)
Let $0<a<b$.
Suppose $(X_{i},d_{i},m_{i})$ satisfies $\mathrm{MDADF}$ property on $\varphi_{i}^{-1}((a,b))$, then $(X_{\infty},d_{\infty},m_{\infty})$ satisfies $\mathrm{MDADF}$ property on $\varphi_{\infty}^{-1}((a,b))$.
\end{thm}
\begin{rem}
The proof here is modified from the proof of Theorem 29.24 in \cite{Vi09}.
\end{rem}
\begin{proof}[Proof of Theorem \ref{thm4.5}]
It is sufficient to prove that $(X_{\infty},d_{\infty},m_{\infty})$ satisfies $\mathrm{MDADF}$ property on $\varphi_{\infty}^{-1}((a',b'))$ for any $a',b'$ with $a<a'<b'<b$.
By (5) in Theorem \ref{thm4.1}, we will prove the entropy inequality (\ref{H-entropy}) holds for any $N'\geq1$.
In the following, we fix $a',b'$ with $a<a'<b'<b$ and fix $N'>1$.

For every $\mu_{\infty,0}\in \mathcal{P}_{2}(X_{\infty})$ with $\mathrm{supp}(\mu_{\infty,0})\subset \varphi_{\infty}^{-1}([a',b'])$ and $\mu_{\infty,0}\ll m_{\infty}$.
Let $\rho_{\infty,0}:=\frac{d \mu_{\infty,0}}{d m_{\infty}}$ be the density function.

By Theorem C.12 of \cite{LV09}, there exist $\mu^{(j)}_{\infty,0}=\rho^{(j)}m_{\infty}\in \mathcal{P}_{2}(X_{\infty})$, where $\{\rho^{(j)}\}_{j}$ is a sequence of continuous functions such that $\mathrm{supp}(\rho^{(j)})\subset\subset \varphi_{\infty}^{-1}((a,b))$,
and in addition, $\mu^{(j)}_{\infty,0}\rightarrow\mu_{\infty,0}$ weakly,
\begin{align}\label{4.13}
\mathcal{S}_{N'}(\mu^{(j)}_{\infty,0}|m_{\infty}) \rightarrow\mathcal{S}_{N'}(\mu_{\infty,0}|m_{\infty})
\end{align}
when $j\rightarrow\infty$.

Let $a^{(j)}:=\min\{a',\varphi_{\infty}(x)\bigl|x\in\mathrm{supp}(\rho^{(j)})\}$.
Note that $a^{(j)}\rightarrow a'$ as $j\rightarrow\infty$.

For every $\mu^{(j)}_{\infty,0}$, let $[0,1]\ni t\mapsto \mu_{\infty,t}^{(j)}=(F_{t}^{a^{(j)}})_{\#}\mu_{\infty,0}^{(j)}$ be the unique $L^{2}$-Wasserstein geodesic defined in Remark \ref{rem3.12},
and let $\pi^{(j)}_{\infty}:=(\mathrm{Id},F_{1}^{a^{(j)}})_{\#}\mu_{\infty,0}^{(j)}$ be the unique optimal transportation (with respect to the cost function $c(x,y)=\frac{d_{\infty}^{2}(x,y)}{2}$) between $\mu_{\infty,0}^{(j)}$ and $\mu_{\infty,1}^{(j)}$.
Define $\psi_{\infty}^{(j)}:X_{\infty}\rightarrow\mathbb{R}$ to be $$\psi_{\infty}^{(j)}(x)=\frac{(\max\{\varphi_{\infty}(x)-a^{(j)},0\})^{2}}{2}.$$
Then $\psi_{\infty}^{(j)}$ is a Kantorovich potential for the couple $\mu_{\infty,0}^{(j)}$ and $\mu_{\infty,1}^{(j)}$.

Similarly, let $[0,1]\ni t\mapsto \mu_{\infty,t}=(F_{t}^{a'})_{\#}\mu_{\infty,0}$ be the unique $L^{2}$-Wasserstein geodesic defined in Remark \ref{rem3.12}.
Then $\pi_{\infty}:=(\mathrm{Id},F_{1}^{a'})_{\#}\mu_{\infty,0}$ is the unique optimal transportation between $\mu_{\infty,0}$ and $\mu_{\infty,1}$, and $\psi_{\infty}:X_{\infty}\rightarrow\mathbb{R}$ given by $\psi_{\infty}(x)=\frac{(\max\{\varphi_{\infty}(x)-a',0\})^{2}}{2}$ is a Kantorovich potential.

Obviously $\psi_{\infty}^{(j)}\rightarrow \psi_{\infty}$ uniformly on any compact set when $j\rightarrow\infty$.

Let $\Phi_{i}:B_{p_{i}}(\epsilon_{i}^{-1})\rightarrow B_{p_{\infty}}(\epsilon_{i}^{-1})$ be the  $\epsilon_{i}$-Gromov-Hausdorff approximations with $\epsilon_{i}\rightarrow0$ as in Definition \ref{def3.21}.
For each $j$, since $\rho^{(j)}$ is continuous and has compact support,
$$Z_{i}^{(j)}:=\int_{X_{i}}(\rho^{(j)}\circ\Phi_{i})dm_{i}=\int_{X_{\infty}}\rho^{(j)} d(\Phi_{i})_{\#}m_{i}$$
converge to $1$ when $i\rightarrow\infty$.
Define
$$\mu_{i,0}^{(j)}:=\frac{\rho^{(j)}\circ\Phi_{i}}{Z_{i}^{(j)}}m_{i}.$$
For $i$ sufficiently large, we have $\mu_{i,0}^{(j)}\in \mathcal{P}_{2}(X_{i})$ and $\mathrm{supp}(\mu^{(j)}_{i,0})\subset\subset \varphi_{i}^{-1}((a,b))$.
Note that for any bounded continuous function $h:X_{\infty}\rightarrow \mathbb{R}$, we have
\begin{align}
&\int_{X_{\infty}}hd(\Phi_{i})_{\#}\mu_{i,0}^{(j)}=\int_{X_{i}}h\circ \Phi_{i}d\mu_{i,0}^{(j)} =\int_{X_{i}}h\circ \Phi_{i}\frac{\rho^{(j)}\circ\Phi_{i}}{Z_{i}^{(j)}}dm_{i}\\
=&\frac{1}{Z_{i}^{(j)}}\int_{X_{\infty}}h\rho^{(j)} d(\Phi_{i})_{\#}m_{i}\rightarrow \int_{X_{\infty}}h\rho^{(j)} dm_{\infty}=\int_{X_{\infty}}h d\mu^{(j)}_{\infty,0}.\nonumber
\end{align}
Hence for each $j$, $(\Phi_{i})_{\#}\mu_{i,0}^{(j)}$ converges to $\mu^{(j)}_{\infty,0}$ weakly when $i\rightarrow\infty$.

In addition, we have
\begin{align}\label{4.21}
\lim_{i\rightarrow\infty}\mathcal{S}_{N'}(\mu^{(j)}_{i,0}|m_{i})
=&-\lim_{i\rightarrow\infty}\int_{X_{i}}(\frac{\rho^{(j)}\circ\Phi_{i}}{Z_{i}^{(j)}})^{1-\frac{1}{N'}}dm_{i}\\
=&-\lim_{i\rightarrow\infty}\int_{X_{\infty}}(\rho^{(j)})^{1-\frac{1}{N'}}d(\Phi_{i})_{\#}m_{i} \nonumber\\ =&-\int_{X_{\infty}}(\rho^{(j)})^{1-\frac{1}{N'}}dm_{\infty} \nonumber\\ =&\mathcal{S}_{N'}(\mu^{(j)}_{\infty,0}|m_{\infty}).\nonumber
\end{align}

Let $a_{i}^{(j)}:=\min\{a^{(j)},\varphi_{i}(x)\bigl|x\in\mathrm{supp}(\mu_{i,0}^{(j)})\}$.
Note that $a_{i}^{(j)}\rightarrow a^{(j)}$ as $i\rightarrow\infty$.

For every $\mu_{i,0}^{(j)}$, let $[0,1]\ni t\mapsto \mu_{i,t}^{(j)}=(F_{t}^{a_{i}^{(j)}})_{\#}\mu_{i,0}^{(j)}$ be the unique $L^{2}$-Wasserstein geodesic defined in Remark \ref{rem3.12}.
and let $\pi^{(j)}_{i}:=(\mathrm{Id},F_{1}^{a_{i}^{(j)}})_{\#}\mu_{i,0}^{(j)}$ be the unique optimal transportation (with respect to the cost function $c(x,y)=\frac{d_{i}^{2}(x,y)}{2}$) between $\mu_{i,0}^{(j)}$ and $\mu_{i,1}^{(j)}$.
Define $\psi_{i}^{(j)}:X\rightarrow\mathbb{R}$ by
$$\psi_{i}^{(j)}(x)=\frac{(\max\{\varphi_{i}(x)-a_{i}^{(j)},0\})^{2}}{2}.$$
Then $\psi_{i}^{(j)}$ is a Kantorovich potential for $\mu_{i,0}^{(j)}$ and $\mu_{i,1}^{(j)}$.

Moreover, by Lemma \ref{lem3.01} and Theorem \ref{thm3.25}, it is not hard to see that there exists $A^{(j)}_{i}\subset\mathrm{supp}(\pi^{(j)}_{i})$ such that $\pi^{(j)}_{i}(A^{(j)}_{i})=1$, and for every $(x_{i},y_{i})\in A^{(j)}_{i}$, $y_{i}$ is the unique point satisfying
\begin{align}\label{4.12}
\psi_{i}^{(j)}(x_{i})=\frac{d_{i}^{2}(x_{i},y_{i})}{2},
\end{align}
\begin{align}\label{4.11}
\psi_{i}^{(j)}(y_{i})=0,\text{ and hence }(\psi_{i}^{(j)})^{c}(y_{i})=0.
\end{align}

Since each $(X_{i},d_{i},m_{i})$ satisfies $\mathrm{MDADF}$ property on $\varphi_{i}^{-1}((a,b))$, by Theorem \ref{thm4.1},
\begin{align}\label{4.14}
\mathcal{S}_{N'}(\mu_{i,t}^{(j)}|m_{i}) \leq (1-t) \mathcal{S}_{N'}(\mu_{i,0}^{(j)}|m_{i})
\end{align}
holds for every $t\in[0,1]$.

By Theorem 28.9 of \cite{Vi09}, for each $j$, after extracting a subsequence of $i$, still denoted by $i$ for simplicity, there exist a $L^{2}$-Wasserstein geodesic $[0,1]\ni t\mapsto\tilde{\mu}^{(j)}_{\infty,t}$
and an optimal transportation $\tilde{\pi}^{(j)}_{\infty}$ of $\tilde{\mu}^{(j)}_{\infty,0}$ and $\tilde{\mu}^{(j)}_{\infty,1}$ such that
\begin{align}\label{4.15}
\lim_{i\rightarrow\infty}(\Phi_{i},\Phi_{i})_{\#}\pi^{(j)}_{i}=\tilde{\pi}^{(j)}_{\infty}&\text{ weakly in } X_{\infty}\times X_{\infty},
\end{align}

\begin{align}\label{4.17}
\lim_{i\rightarrow\infty} \sup_{0\leq t \leq1}W_{2}((\Phi_{i})_{\#}\mu_{i,t}^{(j)},\tilde{\mu}^{(j)}_{\infty,t})=0.
\end{align}

Note that from the above constructions, we have $\tilde{\mu}^{(j)}_{\infty,0}=\mu^{(j)}_{\infty,0}$.

\begin{claim}\label{claim3.1}
$\mathrm{supp}(\tilde{\pi}^{(j)}_{\infty})\subset \partial^{c}\psi_{\infty}^{(j)}$.
\end{claim}
\begin{proof}[Proof of Claim \ref{claim3.1}]
For any $(x,y)\in\mathrm{supp}(\tilde{\pi}^{(j)}_{\infty})$, by (\ref{4.15}), it is easy to see that there exist $(x_{i},y_{i})\in A^{(j)}_{i}$ such that $(x_{i},y_{i})\xrightarrow{GH}(x,y)$.
Note that when $i\rightarrow\infty$, $\varphi_{i}$ converge to $\varphi_{\infty}$ uniformly on any compact set and $a_{i}^{(j)}\rightarrow a^{(j)}$, thus $\psi_{i}^{(j)}$ converge to $\psi_{\infty}^{(j)}$  uniformly on any compact set.
Hence by (\ref{4.12}) and (\ref{4.11}), we have
\begin{align}
\psi_{\infty}^{(j)}(x)=\frac{1}{2}d_{\infty}^{2}(x,y),
\end{align}
\begin{align}
\psi_{\infty}^{(j)}(y)=0,\text{ and hence }(\psi_{\infty}^{(j)})^{c}(y)=0.
\end{align}

Thus $\mathrm{supp}(\tilde{\pi}^{(j)}_{\infty})\subset \partial^{c}\psi_{\infty}^{(j)}$.
\end{proof}
By Claim \ref{claim3.1}, Lemma \ref{lem3.01} and Theorem \ref{thm3.25}, $\tilde{\pi}^{(j)}_{\infty}=(\mathrm{Id},F_{1}^{a^{(j)}})_{\#}\tilde{\mu}_{\infty,0}^{(j)}=\pi^{(j)}_{\infty}$, and $\tilde{\mu}^{(j)}_{\infty,t}=\mu^{(j)}_{\infty,t}$ hold for every $t\in[0,1]$.
In particular, for every $t\in[0,1]$, $(\Phi_{i})_{\#}\mu_{i,t}^{(j)}$ converge to $\mu^{(j)}_{\infty,t}$ weakly when $i\rightarrow\infty$.

Thus we have
\begin{align}\label{4.20}
\mathcal{S}_{N'}(\mu^{(j)}_{\infty,t}|m_{\infty})&\leq \liminf_{i\rightarrow\infty} \mathcal{S}_{N'}((\Phi_{i})_{\#}\mu^{(j)}_{i,t}|(\Phi_{i})_{\#}m_{i})  \\
&\leq\liminf_{i\rightarrow\infty} \mathcal{S}_{N'}(\mu^{(j)}_{i,t}|m_{i}),\nonumber
\end{align}
where in the first inequality we use the weakly lower semicontinuity of $\mathcal{S}_{N'}$ (see Theorem 29.20 in \cite{Vi09}), and in the second inequality we use the property that $\mathcal{S}_{N'}$ never increased by push-forward (see also Theorem 29.20 in \cite{Vi09}).

Note that the geodesic $[0,1]\ni t\mapsto{\mu}^{(j)}_{\infty,t}$ is $D$-Lipschitz with $D$ a constant depending on the diameter of $\varphi_{\infty}^{-1}(a,b)$ (and hence independent of $j$) and all these geodesics are contained in a compact subset of $\mathcal{P}_{2}(X_{\infty})$.
By Arzela-Ascoli Theorem, there exist a $L^{2}$-Wasserstein geodesic $[0,1]\ni t\mapsto\bar{\mu}_{\infty,t}$ and a subsequence of $j$ (still denoted by $j$ for simplicity) such that
\begin{align}
\sup_{0\leq t \leq1}W_{2}(\mu^{(j)}_{\infty,t},\bar{\mu}_{\infty,t})\rightarrow 0,
\end{align}
\begin{align}
\pi^{(j)}_{\infty}\rightarrow\bar{\pi}_{\infty}&\text{ weakly in } X_{\infty}\times X_{\infty},
\end{align}
when $j\rightarrow\infty$, where $\bar{\pi}_{\infty}$ is a optimal coupling of $\bar{\mu}_{\infty,0}$ and $\bar{\mu}_{\infty,1}$.
Note that by the constructions, $\bar{\mu}_{\infty,0}=\mu_{\infty,0}$.

Argue similar to the proof of Claim \ref{claim3.1}, we can prove that
$\mathrm{supp}(\bar{\pi}_{\infty})\subset \partial^{c}\psi_{\infty}$.
Then by the uniqueness of optimal transport, we have $\bar{\pi}_{\infty}=(\mathrm{Id},F_{1}^{a'})_{\#}\bar{\mu}_{\infty,0}=\pi_{\infty}$, and $\bar{\mu}_{\infty,t}=\mu_{\infty,t}$ for every $t\in[0,1]$.
In particular, $\mu^{(j)}_{\infty,t}$ converge to $\mu_{\infty,t}$ weakly for every $t\in[0,1]$.
Hence by the weakly lower semicontinuity of $\mathcal{S}_{N'}$,
\begin{align}\label{4.19}
\mathcal{S}_{N'}(\mu_{\infty,t}|m_{\infty})  \leq\liminf_{j\rightarrow\infty} \mathcal{S}_{N'}(\mu^{(j)}_{\infty,t}|m_{\infty})
\end{align}
holds for every $t\in[0,1]$.

By (\ref{4.13}), (\ref{4.21}), (\ref{4.14}), (\ref{4.20}) and (\ref{4.19}), we obtain
\begin{align}\label{4.16}
\mathcal{S}_{N'}(\mu_{\infty,t}|m_{\infty}) \leq (1-t) \mathcal{S}_{N'}(\mu_{\infty,0}|m_{\infty}).
\end{align}

Because (\ref{4.16}) holds for every $N'>1$, let $N'\downarrow1$, we know (\ref{4.16}) holds for $N'=1$.
Thus $(X_{\infty},d_{\infty},m_{\infty})$ satisfies $\mathrm{MDADF}$ property on $\varphi_{\infty}^{-1}((a',b'))$.
The proof is completed.
\end{proof}

\subsection{An additional property}\label{sec4.4}
\begin{prop}\label{prop4.8}
Suppose $(X,d,m)$ is a $\mathrm{RCD}(0,N)$ space, $E$ is a closed subset such that $\mathrm{diam}(\partial E)<\infty$.
Let $\varphi(x)=d(x,E)$ be the distance function to $E$.
If $(X,d,m)$ satisfies $\mathrm{MDADF}$ property on $\varphi^{-1}((a,b))$ for $0< a<b$, then $(X,d,m)$ satisfies $\mathrm{MDADF}$ property on $\varphi^{-1}((a,l))$ for any $l>a$.
\end{prop}

\begin{proof}
The case $l\in(a,b]$ is obvious.
In the following let $l>b$ be any fixed number.
Then we consider the push forward measure $\mathfrak{q}=\mathcal{Q}_{\#}m\llcorner_{\mathcal{T}^{l}}$ as in Remark \ref{rem3.15}, and let $h(q,t)$ be given as in  Theorem \ref{thm3.15}..
Because $(X,d,m)$ is a $\mathrm{RCD}(0,N)$ space and satisfies $\mathrm{MDADF}$ property on $\varphi^{-1}((a,b))$, by Theorems \ref{thm3.15} and \ref{thm4.1}, there exists $\tilde{Q}\subset Q$ such that:
\begin{enumerate}
  \item $\mathfrak{q}(Q\setminus\tilde{Q})=0$,
  \item for every $q\in \tilde{Q}$, $h(q,\cdot)^{\frac{1}{N-1}}$ is a concave function on $\mathrm{Dom}(g(q,\cdot))\cap(a,l)$,
  \item for every $q\in \tilde{Q}$, $h(q,\cdot)^{\frac{1}{N-1}}$ is a non-increasing function on $\mathrm{Dom}(g(q,\cdot))\cap(a,b)$.
\end{enumerate}
By (2) and (3), we obtain that for every $q\in \tilde{Q}$, $h(q,\cdot)^{\frac{1}{N-1}}$ is a non-increasing function on $\mathrm{Dom}(g(q,\cdot))\cap(a,l)$.
Hence by Theorem \ref{thm4.1}, we know $(X,d,m)$ satisfies $\mathrm{MDADF}$ property on $\varphi^{-1}((a,l))$.
\end{proof}

\begin{rem}
By Proposition \ref{prop4.8}, it is obvious that if we replace the assumption (1) in the statement of Theorem \ref{main-1.1} by
\begin{description}
  \item[(1)'] there exists an $\epsilon>0$ such that $(X,d,m)$ satisfies $\mathrm{MDADF}$ property in $\varphi^{-1}((a,a+\epsilon))$,
\end{description}
then the same conclusion in Theorem \ref{main-1.1} holds.
\end{rem}

\section{Proof of Theorem \ref{main-1.1} and some corollaries}\label{sec5}
\begin{proof}[Proof of Theorem \ref{main-1.1}:]
Argue by contradiction.
Suppose for $0<a<c<b$, $\bar{c}>0$, $0<\alpha'<\alpha<\frac{b-a}{2}$, there exist a positive constant $\Psi_{0}$ and a sequence of positive numbers $\{\delta_{i}\}$ with $\delta_{i}\rightarrow0$ satisfying the following.
\begin{enumerate}
  \item There exist a sequence of $\mathrm{RCD}(0,N)$ spaces $(X_{i},d_{i},m_{i})$ and a closed subset $E_{i}\subset X_{i}$ such that $\sup_{i}\mathrm{diam}(\partial E_{i})\leq\bar{c}$.
  \item It holds
  \begin{align}\label{4.22221}
  \frac{m_{i}(\varphi_{i}^{-1}([a,c]))}{m_{i}(\varphi_{i}^{-1}([a,b]))}\leq(1+\delta_{i})\frac{c-a}{b-a},
  \end{align}
  where $\varphi_{i}(x)=d_{i}(x,E_{i})$.
  \item $(X_{i},d_{i},m_{i})$ satisfies $\mathrm{MDADF}$ property on $\varphi_{i}^{-1}((a,b))$.
  \item Let $(Z,d_{Z})$ be any length extended metric space with at most $C_{0}$ components and every component has diameter (with respect to $d_{Z}$) bounded from above by $C_{1}$. (The constants $C_{0}$ and $C_{1}$ will be determined in Claims \ref{claim4.6} and \ref{claim4.7} respectively.)
      Let $Y=Z\times(a+\alpha,b-\alpha)$ be equipped with the product distance $d_{Y}$, $r:Y\rightarrow(a+\alpha,b-\alpha)$ be the projection to the second factor.
      Suppose $\Phi:\varphi_{i}^{-1}((a+\alpha,b-\alpha))\rightarrow Y$ is a map such that
      \begin{align}
      |\varphi_{i}(x)-r(\Phi(x))|<\Psi_{0}
      \end{align}
      holds for every $x\in \varphi_{i}^{-1}((a+\alpha,b-\alpha))$, then $\Phi$ is not a $\Psi_{0}$-Gromov-Hausdorff approximation between $(\varphi_{i}^{-1}((a+\alpha,b-\alpha)),d^{\alpha,\alpha'})$ and $(Z\times(a+\alpha,b-\alpha)),d_{Y})$.
\end{enumerate}

We choose $p_{i}\in \partial E_{i}$.
Without loss of generality, we assume $m_{i}(B_{p_{i}}(1))=1$ and $(X_{i},p_{i},d_{i},m_{i}) \xrightarrow{pmGH} (X_{\infty},p_{\infty},d_{\infty},m_{\infty})$ with $(X_{\infty},d_{\infty},m_{\infty})$ a  $\mathrm{RCD}(0,N)$ space,
and assume $\varphi_{i}$ converge uniformly on any compact subsets to a $1$-Lipschitz function $\varphi_{\infty}:X_{\infty}\rightarrow[0,\infty)$ with $\varphi_{\infty}(p_{\infty})=0$.

By Lemma \ref{lem3.05}, $\varphi_{\infty}(x)=d_{\infty}(x,E_{\infty})$, where $E_{\infty}=\{x\in X_{\infty}\bigl|\varphi_{\infty}(x)=0\}$.
In particular, $|D\varphi_{\infty}|(x)=1$ holds for $m$-a.e. $x\in \varphi_{\infty}^{-1}((a,b))$.

By Lemma \ref{lem3.06} and (\ref{4.22221}), we have
\begin{align} \label{4.22222}
  \frac{m_{\infty}(\varphi_{\infty}^{-1}([a,c]))}{m_{\infty}(\varphi_{\infty}^{-1}([a,b]))}\leq \frac{c-a}{b-a}.
\end{align}

On the other hand, by assumption (3) and Theorem \ref{thm4.5}, we know $(X_{\infty},d_{\infty},m_{\infty})$ satisfies $\mathrm{MDADF}$ property on $\varphi_{\infty}^{-1}((a,b))$.
Then by Theorem \ref{thm4.1} it is easy to derive
\begin{align} \label{4.22223}
  \frac{m_{\infty}(\varphi_{\infty}^{-1}([a,c]))}{m_{\infty}(\varphi_{\infty}^{-1}([a,b]))}\geq \frac{c-a}{b-a},
\end{align}
and hence
\begin{align} \label{4.22224}
  \frac{m_{\infty}(\varphi_{\infty}^{-1}([a,c]))}{m_{\infty}(\varphi_{\infty}^{-1}([a,b]))}= \frac{c-a}{b-a}.
\end{align}

By (\ref{disintegration}), Theorem \ref{thm3.15} and (2) in Theorem \ref{thm4.1}, it is easy to obtain

\begin{claim}\label{claim4.1}
There exists $\tilde{Q}\subset Q(\varphi_{\infty}^{-1}([a,b]))$ such that $\mathfrak{q}(Q(\varphi_{\infty}^{-1}([a,b]))\setminus \tilde{Q})=0$, and for any $q\in \tilde{Q}$, $(a,b)\subset\mathrm{Dom}(g(q,\cdot))$, $h(q,r_{1})= h(q,r_{2})$ holds for any $r_{1},r_{2}\in(a,b)$.
Furthermore,
\begin{align} \label{4.22225}
  \frac{m_{\infty}(\varphi_{\infty}^{-1}([r_{1},r_{2}]))}{m_{\infty}(\varphi_{\infty}^{-1}([a,b]))}= \frac{r_{2}-r_{1}}{b-a}.
\end{align}
holds for any $a<r_{1}< r_{2}<b$.
\end{claim}

Claim \ref{claim4.1} implies that $m_{\infty}\llcorner_{\varphi_{\infty}^{-1}((a,b))}$ is isomorphic to a product measure.
De Philippis and Gigli's proof of `volume cone implies metric cone' (\cite{GigPhi15}) can be modified in this setting to prove that $(\varphi_{\infty}^{-1}((a,b)),d_{\infty})$ is locally isometric to a product metric space.
The proof in \cite{GigPhi15} has provided a complete and clear strategy for such kind of volume rigidity results.
In the following, we give some claims to describe how to proceed by following the strategy of \cite{GigPhi15}.
Some detailed calculations to prove the following claims can be found in \cite{H16}, where the author applied the strategy of \cite{GigPhi15} to prove a rigidity result for the noncompact end of a $\textmd{RCD}(0,N)$ space with strongly minimal volume growth.

Let $F_{t}$ be the flow given by the distance function $\varphi_{\infty}$ as in Definition \ref{def3.13}.
By Claim \ref{claim4.1}, there is an $m$-negligible Borel set $\mathcal{N}\supset \varphi_{\infty}^{-1}((a,b))\setminus\mathcal{T}$ such that for all $t\in[0,b-a)$, $F_{t}:\varphi_{\infty}^{-1}((a+t,b))\setminus \mathcal{N}\rightarrow \varphi_{\infty}^{-1}((a,b-t))\setminus \mathcal{N}$ is a bijection, whose inverse is denoted by  $F_{-t}:=F_{t}^{-1}$, a map from $\varphi_{\infty}^{-1}((a,b-t))\setminus \mathcal{N}$ to $\varphi_{\infty}^{-1}((a+t,b))\setminus \mathcal{N}$.
By Claim \ref{claim4.1}, for any $t\in[0,b-a)$, both $F_{t}:(\varphi_{\infty}^{-1}((a+t,b)),m_{\infty})\rightarrow(\varphi_{\infty}^{-1}((a,b-t)),m_{\infty})$ and $F_{-t}:(\varphi_{\infty}^{-1}((a,b-t)),m_{\infty})\rightarrow(\varphi_{\infty}^{-1}((a+t,b)),m_{\infty})$ are measure-preserving.

Since $(X_{\infty},d_{\infty},m_{\infty})$ satisfies $\mathrm{MDADF}$ property on $\varphi_{\infty}^{-1}((a,b))$, by Corollary \ref{cor4.5}, we have $\mathbf{\Delta}\varphi_{\infty}\leq 0$ on $\varphi_{\infty}^{-1}((a,b))$.
Combining this with Claim \ref{claim4.1}, we have:

\begin{claim}\label{claim4.2}
$\mathbf{\Delta} \varphi_{\infty} =0$ on $\varphi_{\infty}^{-1}((a,b))$.
\end{claim}
\begin{proof}
Let $\phi: \mathbb{R}^{+}\rightarrow[0,1]$ be a Lipschitz function with $\textmd{supp}(\phi)\subset\subset(a,b)$, then we have
\begin{align}
&\int_{\varphi_{\infty}^{-1}((a,b))}\phi(\varphi_{\infty}(x))d\mathbf{\Delta}\varphi_{\infty} \nonumber\\
=&-\int_{\varphi_{\infty}^{-1}((a,b))}\langle D\phi(\varphi_{\infty}(x)),D \varphi_{\infty}(x)\rangle dm(x) \nonumber\\
=&-\int_{\varphi_{\infty}^{-1}((a,b))}\phi'(\varphi_{\infty}(x))|D\varphi_{\infty}(x)|^{2} dm(x) \nonumber\\
=&-\int_{\varphi_{\infty}^{-1}((a,b))}\phi'(\varphi_{\infty}(x)) dm(x) \nonumber\\
=&-\int_{\tilde{Q}}\biggl(\int_{a}^{b}\phi'(s)ds\biggr)h(q,\frac{a+b}{2})d\mathfrak{q}(q)\nonumber\\
=&0. \nonumber
\end{align}
Since $\mathbf{\Delta}\varphi_{\infty}\leq0$ on $\varphi_{\infty}^{-1}((a,b))$ and by the arbitrariness of $\phi$, one can easily derive that $\mathbf{\Delta}\varphi_{\infty}=0$ on $\varphi_{\infty}^{-1}((a,b))$.
\end{proof}

Since $(X_{\infty},d_{\infty},m_{\infty})$ is a $\mathrm{RCD}(0,N)$ space, Bochner inequality holds on it (see \cite{EKS15} \cite{AMS15}).
Combining it with Claim \ref{claim4.2}, we can argue as in Proposition 3.12 of \cite{GigPhi15} to prove the following claim (see also Proposition 5.10 of \cite{H16}):

\begin{claim}\label{claim4.3}
For any $f,g\in D(\Delta)$ such that $f,|Df|,g,|Dg|\in L^{\infty}(X)$, $\Delta f,\Delta g\in W^{1,2}(X)$ and $\mathrm{supp}(f)\subset\subset \varphi_{\infty}^{-1}((a,b))$, we have
\begin{align}\label{5.8}
\int\Delta f\langle D \varphi_{\infty},D g\rangle dm_{\infty}=\int f\langle D \varphi_{\infty},D (\Delta g)\rangle dm_{\infty}.
\end{align}
\end{claim}

As in \cite{GigPhi15}, one can further obtain:

\begin{claim}\label{claim4.4}
$\mathrm{Hess}(\varphi_{\infty})=0$ $m$-a.e. on $\varphi_{\infty}^{-1}((a,b))$.
Here $\mathrm{Hess}(\varphi_{\infty})$ means the Hessian of $\varphi_{\infty}$, see Definition 3.3.1 of \cite{Gig14-2} for details.
\end{claim}

Mainly based on Claim \ref{claim4.3} as well as appropriate cut-off arguments as in Section 3.2 of \cite{GigPhi15}, one can prove that the flow $F_{t}$ preserves the Cheeger energy of a function supported on $\varphi_{\infty}^{-1}((a,b))$.
More precisely, we have:

\begin{claim}\label{claim4.4-1}
Assume $f\in L^{2}(X_{\infty})$ satisfies $\mathrm{supp}(f)\subset \varphi_{\infty}^{-1}((a,\tilde{c}))$ with $\tilde{c}\in(a,b)$, then $f\in W^{1,2}(X_{\infty})$ if and only if $f\circ F_{t}\in W^{1,2}(X_{\infty})$ for any $0\leq t\leq b-\tilde{c}$ and in this case
\begin{align}\label{5.39}
|D(f\circ F_{t} )| = |Df|\circ F_{t}\qquad m_{\infty}\text{-a.e.}
\end{align}
In particular,
\begin{align}\label{5.27}
\mathrm{Ch}(f\circ F_{t})=\mathrm{Ch}(f), \quad \forall 0\leq t\leq b-\tilde{c}.
\end{align}
\end{claim}

By \cite{AGS14}, on the $\textmd{RCD}$ space $(X_{\infty},d_{\infty},m_{\infty})$, any $f\in W^{1,2}(X_{\infty})$ with $|Df|\leq 1$ $m_{\infty}$-a.e. admits a $1$-Lipschitz representative.
Such a property is called the Sobolev-to-Lipschitz property in \cite{Gig13}, and it is a key to deduce metric information from the study of Sobolev functions, see Proposition 4.20 in \cite{Gig13}.

After Claim \ref{claim4.4-1} has been obtained, argue as in Theorem 3.18 of \cite{GigPhi15}, we can prove the following:

\begin{claim}\label{claim4.5}
If we define the map $F:\mathrm{Dom}(F)\subset[0,b-a)\times \varphi_{\infty}^{-1}((a,b))\rightarrow \varphi_{\infty}^{-1}((a,b))$ by $F(t,x)=F_{t}(x)$ for any $t\in[0,b-a)$, $x\in\varphi_{\infty}^{-1}((a+t,b))\setminus \mathcal{N}$, 
then $F$ admits a locally Lipschitz representative with respect to the measure $\mathcal{L}^{1}\otimes m_{\infty}$,
and if we still denote such a representative by $F$, then for every $t\in [0,b-a)$,
$F_{t}$ is invertible and is a locally isometry from $\varphi_{\infty}^{-1}((a+t,b))$ to $\varphi_{\infty}^{-1}((a,b-t))$, i.e. for any $x_{0}\in \varphi_{\infty}^{-1}((a+t,b))$, there exists a geodesic ball $B_{x_{0}}(r)\subset \varphi_{\infty}^{-1}((a+t,b))$ such that the restriction of $F_{t}$ on $B_{x_{0}}(r)$ is an isometry onto its image.
\end{claim}

We remark that in Claim \ref{claim4.5}, the maps $F_{t}$ are local isometry instead of isometry.
This is because in Claim \ref{claim4.4-1}, (\ref{5.39}) holds only for $f\in W^{1,2}(X_{\infty})$ with $\mathrm{supp}(f)\subset \varphi_{\infty}^{-1}((a,b))$.

From now on, when considering the maps $F_{t}$ on $\varphi_{\infty}^{-1}((a,b))$, we always refer to their continuous versions given in Claim \ref{claim4.5}.
And we will denote the inverse of $F_{t}:\varphi_{\infty}^{-1}((a+t,b))\rightarrow\varphi_{\infty}^{-1}((a,b-t))$ by $F_{-t}$.

Denote by $Z=\varphi^{-1}_{\infty}(\frac{a+b}{2})$.
We define the projection map $\mathrm{Pr}:\varphi_{\infty}^{-1}((a, b))\rightarrow Z$ by
$$\mathrm{Pr}(x)=F_{\varphi_{\infty}(x)-\frac{a+b}{2}}(x).$$
By Claim \ref{claim4.5}, the map $\textmd{Pr}$ is well defined and locally Lipschitz.
$Z$ will serve as the `base' space for the product metric measure space, so we will endow it a natural metric measure structure.

Note that by (\ref{4.22225}), $(\varphi_{\infty})_{\#}(m_{\infty}\llcorner_{\varphi_{\infty}^{-1}([a,b])})=\tilde{c}\mathcal{L}^{1} \llcorner_{[a,b]}$, where $\tilde{c}= \frac{m_{\infty}(\varphi_{\infty}^{-1}([a,b]))}{b-a}$.
Now we apply the disintegration Theorem and obtain a disintegration $(a,b)\ni r\mapsto\tilde{m}_{r}\in\mathcal{P}(X_{\infty})$ of $m_{\infty}\llcorner_{\varphi_{\infty}^{-1}([a,b])}$ over $\tilde{c}\mathcal{L}^{1} \llcorner_{[a,b]}$ strongly consistent with $\varphi_{\infty}$.
Moreover, similar to \cite{GigPhi15}, we can prove:
\begin{claim}\label{claim5.1}
The $\tilde{m}_{r}$ can be chosen to be a weakly continuous family, i.e. for any $\varphi\in C_{c}(\varphi_{\infty}^{-1}((a,b)))$, the map $r\mapsto I_{\varphi}(r):=\int\varphi d\tilde{m}_{r}$ is continuous.
Furthermore, for every $r\in(a,b)$, $\tilde{m}_{r}=(F_{t})_{\#}\tilde{m}_{r+t}$ holds for a.e. $t\in(a-r,b-r)$.
\end{claim}

The measure $m_{Z}$ on $Z$ is chosen to be $m_{Z}:=\tilde{c}\tilde{m}_{\frac{a+b}{2}}$.
Then by Claim \ref{claim5.1}, for any $a<r_{1}< r_{2}<b$, it holds
\begin{align}\label{5.5}
(\mathrm{Pr})_{\#}m\llcorner_{\varphi_{\infty}^{-1}([r_{1},r_{2}])} = (r_{2}-r_{1})m_{Z}.
\end{align}

Similar to \cite{GigPhi15}, we endow $Z$ an extended distance $d_{Z}$ by
$$d_{Z}(x', y'):=\inf_{\sigma}\biggl(\int_{0}^{1}|\dot{\sigma}_{t}|^{2} dt\biggr)^{\frac{1}{2}}$$
for $x', y'\in Z$,
where the infimum is taken among all Lipschitz curves $\sigma: [0, 1]\rightarrow Z \subset X_{\infty}$ connecting $x', y'$.
The metric speed in the above definition is computed with respect to the distance $d_{\infty}$.

Note that in \cite{GigPhi15}, the authors consider the distance function to a point, and the uniqueness of optimal transportation on the underlying space will imply that a level set of the distance function has only one path-connected component unless the level set consists of two points (see Section 3.4 of \cite{GigPhi15} for details).
In \cite{H16}, due to the geometric assumptions, a level set $Z$ also has exactly one path-connected component (see Corollary 5.20 of \cite{H16}).
However, in general case, $Z$ may contain more than one components.
The following claim will imply that the number of components of $Z$ has an upper bound.

\begin{claim}\label{claim4.6}
There is a constant $C_{0}$ depending only on $N,\bar{c},a,b$ such that $\varphi_{\infty}^{-1}((a, b))$ has at most $C_{0}$ path-connected components.
\end{claim}

\begin{proof}
Let $l_{1}=\min{\{1,\frac{b-a}{2}\}}$, $l_{2}=b+\bar{c}+1$.
Let $\{U_{j}\}_{j=1}^{k}$ be a collection of path-connected components of $\varphi_{\infty}^{-1}((a,b))$.
Note that every $U_{j}$ is open, hence $m_{\infty}(U_{j})>0$ for every $j$.
Then by Claim \ref{claim4.1}, it is easy to see that there is at least a $q\in \tilde{Q}$ such that $(a,b)\subset\mathrm{Dom}(g(q,\cdot))$ and $g(q,\cdot)\cap U_{j}\neq\emptyset$.
In particular, $U_{j}\cap\varphi^{-1}_{\infty}(\frac{a+b}{2})$ is not empty.
In the following we fix $z_{j}\in U_{j}\cap\varphi^{-1}_{\infty}(\frac{a+b}{2})$ for each $j$.
It is easy to see that $\varphi_{\infty}^{-1}((a,b))\subset B_{p_{\infty}}(l_{2})$ and $B_{p_{\infty}}(1)\subset B_{z_{j}}(l_{2})$ for every $j$.
By volume comparison, we have
\begin{align}\label{5.10}
m_{\infty}(B_{p_{\infty}}(l_{2}))\leq l_{2}^{N},
\end{align}
\begin{align}\label{5.11}
m_{\infty}(B_{z_{j}}(l_{1}))\geq \frac{l_{1}^{N}}{l_{2}^{N}}m_{\infty}(B_{z_{j}}(l_{2})) \geq \frac{l_{1}^{N}}{l_{2}^{N}}.
\end{align}
On the other hand, it is not hard to check that the geodesic balls $B_{z_{j}}(l_{1})$ are disjoint to each other, and all $B_{z_{j}}(l_{1})$ are contained in $\varphi_{\infty}^{-1}((a,b))$, hence
\begin{align}\label{5.12}
\sum_{j=1}^{k}m_{\infty}(B_{z_{j}}(l_{1}))\leq m_{\infty}(\varphi_{\infty}^{-1}((a,b))) \leq m_{\infty}(B_{p_{\infty}}(l_{2})).
\end{align}
By (\ref{5.10}), (\ref{5.11}), (\ref{5.12}), we obtain $k\leq\frac{l_{2}^{2N}}{l_{1}^{N}}$.
In other words, $\varphi_{\infty}^{-1}((a,b))$ has at most $\frac{l_{2}^{2N}}{l_{1}^{N}}$ path-connected components.
\end{proof}

Denote by $Z^{(j)}=Z\cap U^{(j)}$, where $\{U^{(j)}\}_{j=1}^{k}$ consists of all the path-connected components of $\varphi_{\infty}^{-1}((a, b))$.

For any $x',y'\in Z^{(j)}$, there exists a Lipschitz curve $\eta:[0,1]\rightarrow U^{(j)}$ with $\eta_{0}=x'$, $\eta_{1}=y'$.
The curve $\sigma=\mathrm{Pr}(\eta)$ is a Lipschitz curve connecting $x'$ and $y'$ in $Z^{(j)}$.
In conclusion, we have

\begin{claim}\label{claim4.8}
For every $x', y' \in Z^{(j)}$, there is a Lipschitz curve $\sigma:[0,1]\rightarrow Z^{(j)}$ with $\sigma_{0}=x'$, $\sigma_{1}=y'$.
In particular, every $Z^{(j)}$ is path-connected.
Moreover, $d_{Z}(x', y')=\infty$ if and only if $x'$ and $y'$ do not lie in the same $Z^{(j)}$.
\end{claim}

Denote by $Y=Z\times (a,b)$, we endow $Y$ with the product measure $m_{Z}\otimes \mathcal{L}^{1}$ and the product extended distance $d_{Z}\times d_{\mathrm{Eucl}}$.
We denote $(Y,d_{Z}\times d_{\mathrm{Eucl}},m_{Z}\otimes\mathcal{L}^{1})$ by $(Y,d_{Y},m_{Y})$ for simplicity.

Define two maps $S:\varphi_{\infty}^{-1}((a, b))\rightarrow Y$ and $T:Y\rightarrow \varphi_{\infty}^{-1}((a, b))$ by
\begin{align}\label{5.23}
S(x)=(\mathrm{Pr}(x),\varphi_{\infty}(x))\quad\text{for}\quad x \in \varphi_{\infty}^{-1}((a, b)),
\end{align}
\begin{align}\label{5.22}
T(z,t)=F_{\frac{a+b}{2}-t}(z) \quad\text{for}\quad(z,t)\in Y.
\end{align}

It's easy to see that $S\circ T=\mathrm{Id}|_{Y}$, $T\circ S=\mathrm{Id}|_{\varphi_{\infty}^{-1}((a, b))}$, both $S$ and $T$ are locally Lipschitz and
\begin{align}\label{5.43}
T_{\#}m_{Y}=m\llcorner_{\varphi_{\infty}^{-1}((a, b))},\qquad \text{ and }\qquad S_{\#}m\llcorner_{\varphi_{\infty}^{-1}((a, b))}=m_{Y}.
\end{align}

In fact, we can prove
\begin{claim}\label{claim4.9}
The maps $S$ and $T$ are both local isometries.
\end{claim}

Claim \ref{claim4.9} is easy to verified if there is some $j$ such that $Z^{(j)}$ consists of exactly one point.
In this case, it is easy to see that $U^{(j)}$ is a geodesic ray.
Then by Theorem 1.1 in \cite{KL16}, the whole space $(X_{\infty},d_{\infty})$ is isometric to either $(\mathbb{R},d_{\mathrm{Eucl}})$, or $(\mathbb{R}^{+},d_{\mathrm{Eucl}})$, or $([0,l],d_{\mathrm{Eucl}})$ or $(S^{1}(r),d_{S^{1}(r)})$ for $r>0$.
Here $S^{1}(r):=\{x\in\mathbb{R}^{2}\bigl||x|=r\}$, and $d_{S^{1}(r)}$ is the distance induced by the standard metric.
In addition, it is easy to see that $(\varphi_{\infty}^{-1}((a, b)),d_{\infty}^{0})$ is isometric to a disjoint unit of several $((a,b),d_{\mathrm{Eucl}})$, where $d_{\infty}^{0}$ is the arc-length distance in Definition \ref{def3.17}.

To prove the general case of Claim \ref{claim4.9}, one can repeat the proof of \cite{GigPhi15} with minor modification.
The proof is lengthy and non-trivial, it makes use of some tools and results developed in \cite{Gig14-2}, \cite{GigHan15} etc.
The details are omitted, but we give a rough description here.

The key point is to prove that the maps $S$ and $T$ preserves Cheeger energies similar to Claim \ref{claim4.4-1}, and then we can obtain information on metrics by making use of the Sobolev-to-Lipschitz property.

To compare the minimal weak upper gradients $|D(f\circ T)|_{Y}$ with $|Df|_{X_{\infty}}$ for $f\in W^{1,2}(\varphi_{\infty}^{-1}((a, b)))$, it turn out that the most important point is to prove that   $|Df|_{X_{\infty}}(x)=|Dg|_{Z}(\mathrm{Pr}(x))$ for a function $f$ such that $f(x)=g(\mathrm{Pr}(x))$, $x\in\varphi_{\infty}^{-1}((a, b))$, where $g\in W^{1,2}(Z)$.
By an equivalent formulation of weak upper gradient (see Remark 5.8 of \cite{AGS14II}), we need to compare the metric speed of curves $\eta:[0,1]\rightarrow \varphi_{\infty}^{-1}((a, b))$ and the metric speed of curves $\mathrm{Pr}\circ\eta:[0,1]\rightarrow Z$.
Recall that by Theorem 2.3.18 of \cite{Gig14-2}, there is a link between the point-wise norm of the vector fields and the metric speed of curves; on the other hand, as in \cite{GigPhi15}, $\mathrm{Hess}(\varphi_{\infty})$ is the derivative of norm of vector fields in some sense.
Combining these facts with Claim \ref{claim4.4}, repeat the arguments in \cite{GigPhi15}, we can prove:
\begin{claim}\label{claim4.11}
Let $[c,d]\subset(a,b)$ and $\pi$ be a test plan on $X_{\infty}$ such that $\varphi_{\infty}(\eta_{t})\in[c, d]$ for every $t\in[0, 1]$ and $\pi$-a.e. $\eta$.
Then for $\pi$-a.e. $\eta$, the curve $\tilde{\eta}:=\mathrm{Pr}\circ\eta$ is absolutely continuous and
$|\dot{\tilde{\eta}}_{t}|\leq|\dot{\eta}_{t}|$ holds for a.e. $t\in[0,1]$.
\end{claim}
With Claim \ref{claim4.11}, we can finally prove the maps $S$ and $T$ preserves Cheeger energies.

We also note that in the proof of Claim \ref{claim4.9}, we don't know whether $Y$ is a $\mathrm{RCD}$ space, thus the Sobolev-to-Lipschitz property of $Y$ is not obvious.
To prove $Y$ satisfies the Sobolev-to-Lipschitz property, one need to prove that $(Z,d_{Z},m_{Z})$ is doubling and is a measured-length space in the sense of  \cite{GigHan15} and then apply some results of \cite{GigHan15}.

By Claim \ref{claim4.9}, $(\varphi_{i}^{-1}((a+\alpha,b-\alpha)),d_{i})$ is locally Gromov-Hausdorff close to a product space.
In order to obtain Gromov-Hausdorff closeness of $(\varphi_{i}^{-1}((a+\alpha,b-\alpha)), d_{i}^{\alpha,\alpha'})$ and the product space, we need the following claim.

\begin{claim}\label{claim4.7}
There is a constant $C_{1}$ depending only on $N,\bar{c},a,b$ such that $\mathrm{diam}(Z^{(j)})\leq C_{1}$ for every $j$, here the diameter is computed with respect to $d_{Z}$.
\end{claim}

\begin{proof}
Without loss of generality we only consider $Z^{(1)}$.
Let $l_{1}=\min{\{1,\frac{b-a}{4}\}}$, $l_{2}=b+\bar{c}+1$.
For any $\{z_{i}\}_{i=1}^{T}\subset Z^{(1)}$ such that $B^{d_{Z}}_{z_{i}}(l_{1})$, $i=1,\ldots,T$, form a maximal set of disjoint geodesic balls with respect to $d_{Z}$ on $Z^{(1)}$.
By Lemma \ref{lem3.18} and Claim \ref{claim4.7}, it is easy to see that $B^{d_{Z}}_{z_{i}}(l_{1})=B_{z_{i}}(l_{1})\cap Z^{(1)}$, where $B_{z_{i}}(l_{1})$ are geodesic balls with respect to $d_{\infty}$ on $X_{\infty}$; and in addition, $B_{z_{i}}(l_{1})$, $i=1,\ldots,T$, are disjoint to each other.
Similar to the proof of Claim \ref{claim4.6}, we have
\begin{align}\label{5.17}
\sum_{i=1}^{T}m_{\infty}(B_{z_{i}}(l_{1}))\leq m_{\infty}(\varphi_{\infty}^{-1}((a,b))) \leq m_{\infty}(B_{p_{\infty}}(l_{2}))\leq l_{2}^{N},
\end{align}
\begin{align}\label{5.16}
m_{\infty}(B_{z_{i}}(l_{1}))\geq \frac{l_{1}^{N}}{l_{2}^{N}}m_{\infty}(B_{z_{i}}(l_{2})) \geq \frac{l_{1}^{N}}{l_{2}^{N}}.
\end{align}
Hence $T\leq\frac{l_{2}^{2N}}{l_{1}^{N}}$.
Finally it is not hard to prove that the diameter of $Z^{(1)}$ with respect to $d_{Z}$ is bounded from above by $C_{1}\leq 4Tl_{1}\leq \frac{4l_{2}^{2N}}{l_{1}^{N-1}}$.
\end{proof}

By the assumptions on $X_{i}$, $X_{\infty}$, $\varphi_{i}$ and $\varphi_{\infty}$ at the beginning of the proof, it is easy to see that we may assume that there are two sequences of positive numbers $\{\epsilon_{i}\}$, $\{\delta_{i}\}$ such that $\epsilon_{i}\rightarrow0$, $\delta_{i}\rightarrow0$ and
\begin{description}
  \item[(A)] for every $i$, there exits an $\epsilon_{i}$-Gromov-Hausdorff approximation $\Phi_{i}:(B_{p_{i}}(\epsilon_{i}^{-1}),d_{i})\rightarrow (B_{p_{\infty}}(\epsilon_{i}^{-1}),d_{\infty})$;
  \item[(B)] for every $i$, $|\varphi_{i}(x_{i})-\varphi_{\infty}(x_{\infty})|<\delta_{i}$ holds whenever  $x_{\infty}\in \varphi_{\infty}^{-1}([a,b])$, $x_{i}\in X_{i}$ and $d_{\infty}(\Phi_{i}(x_{i}),x_{\infty})< \epsilon_{i}$.
\end{description}

In general $\Phi_{i}(\varphi_{i}^{-1}((a+\alpha,b-\alpha)))$ is not contained in $\varphi_{\infty}^{-1}((a+\alpha,b-\alpha))$.
Note that by \textbf{(B)}, for any $x\in \varphi_{i}^{-1}((a+\alpha,b-\alpha))$, it holds $\varphi_{\infty}(\Phi_{i}(x))\in(\varphi_{i}(x)-\delta_{i},\varphi_{i}(x)+\delta_{i})\subset (a+\alpha-\delta_{i},b-\alpha+\delta_{i})$.
Define $\hat{\Phi}_{i}:B_{p_{i}}(\epsilon_{i}^{-1})\rightarrow B_{p_{\infty}}(\epsilon_{i}^{-1})$ by
\begin{align}
\hat{\Phi}_{i}(x)=
\left\{
  \begin{array}{ll}
    F_{-\delta_{i}}(\Phi_{i}(x)), & \hbox{if $x\in \varphi_{i}^{-1}((a+\alpha,b-\alpha))$ and $\varphi_{\infty}(\Phi_{i}(x))<a+\alpha$,} \\
    F_{\delta_{i}}(\Phi_{i}(x)), & \hbox{if $x\in \varphi_{i}^{-1}((a+\alpha,b-\alpha))$ and $\varphi_{\infty}(\Phi_{i}(x))>b-\alpha$,} \\
    \Phi_{i}(x), & \hbox{otherwise,}
  \end{array}
\right.
\end{align}
where $F_{\delta_{i}}$ and $F_{-\delta_{i}}$ are maps in Claim \ref{claim4.5}.
Obviously, $\hat{\Phi}_{i}(\varphi_{i}^{-1}((a+\alpha,b-\alpha)))\subset\varphi_{\infty}^{-1}((a+\alpha,b-\alpha))$, and $d_{\infty}(\hat{\Phi}_{i}(x),\Phi_{i}(x))\leq \delta_{i}$.
It is easy to check that for each $i$, both $\hat{\Phi}_{i}$ and its restriction to $\varphi_{i}^{-1}((a+\alpha,b-\alpha))$ are $(\epsilon_{i}+2\delta_{i})$-Gromov-Hausdorff approximations.
Thus if we replace $\Phi_{i}$ in \textbf{(A)} by $\hat{\Phi}_{i}$, replace $\epsilon_{i}$, $\delta_{i}$ suitably and extract a suitable subsequence of $i$, we may assume that except the assumptions \textbf{(A)}, \textbf{(B)}, $\Phi_{i}$ satisfies
\begin{description}
  \item[(C)] for every $i$, $\Phi_{i}(\varphi_{i}^{-1}((a+\alpha,b-\alpha)))\subset \varphi_{\infty}^{-1}((a+\alpha,b-\alpha))$, and $\Phi_{i}:(\varphi_{i}^{-1}((a+\alpha,b-\alpha)), d_{i}) \rightarrow (\varphi_{\infty}^{-1}((a+\alpha,b-\alpha)),d_{\infty})$ is also an $\epsilon_{i}$-Gromov-Hausdorff approximation.
\end{description}
Furthermore, it is easy to see that, after a suitable choice of the parameters, we may further assume
\begin{description}
  \item[(D)] for every $i$, there exits a $4\epsilon_{i}$-Gromov-Hausdorff approximation  $\tilde{\Phi}_{i}:B_{p_{\infty}}(\epsilon_{i}^{-1}) \rightarrow B_{p_{i}}(\epsilon_{i}^{-1})$ which is an $\epsilon_{i}$-inverse of $\Phi_{i}$ and satisfies $\tilde{\Phi}_{i}(\varphi_{\infty}^{-1}((a+\alpha,b-\alpha)))\subset\varphi_{i}^{-1}((a+\alpha,b-\alpha))$.
\end{description}

\begin{claim}\label{claim4.10}
There is a sequence of positive numbers $\{\tilde{\epsilon}_{i}\}$ with $\tilde{\epsilon}_{i}\rightarrow0$ such that for every $i$,
$S\circ\Phi_{i}:(\varphi_{i}^{-1}((a+\alpha,b-\alpha)),d_{i}^{\alpha,\alpha'})\rightarrow (Z\times (a+\alpha,b-\alpha),d_{Y})$ is an $\tilde{\epsilon}_{i}$-Gromov-Hausdorff approximation.
\end{claim}

\begin{proof}
Because $S$ is a local isometry, by Lemma \ref{lem3.18}, one can derive that there is a $r_{0}$ depending only on $\alpha'$ such that for every $x\in \varphi^{-1}_{\infty}([a+\frac{\alpha'}{2},b-\frac{\alpha'}{2}])$, $S$ is an isometry from $B_{x}(r_{0})$ onto its image.
We can further assume $r_{0}<\min\{\frac{\alpha-\alpha'}{4},\frac{\alpha'}{4}\}$.
From the $\epsilon_{i}$-almost onto property of $\Phi_{i}$, we can easily obtain the $\epsilon_{i}$-almost onto property of $S\circ\Phi_{i}$ (for $i$ sufficiently large).
In the following we prove the $\tilde{\epsilon}_{i}$-almost distance preserving property.

(1). For every $x_{1},x_{2}\in \varphi_{i}^{-1}((a+\alpha,b-\alpha))$, let $\gamma:[0,1]\rightarrow Z\times (a+\alpha,b-\alpha)$ be a geodesic connecting $S(\Phi_{i}(x_{1}))$ and $S(\Phi_{i}(x_{2}))$.
By Claim \ref{claim4.7}, $d_{Y}(S(\Phi_{i}(x_{1})),S(\Phi_{i}(x_{2})))=\mathrm{Length}(\gamma)$ is not large than $C_{2}:=\sqrt{C_{1}^{2}+(b-a-2\alpha)^{2}}$.
Let $C_{3}:=\lfloor\frac{C_{2}}{r_{0}}\rfloor+1$, $t_{j}:=\frac{j}{C_{3}}$ for $j=0,1,\ldots,C_{3}$, then we have $d_{\infty}(T(\gamma_{t_{j}}),T(\gamma_{t_{j+1}}))=\mathrm{Length}(\gamma|_{[t_{j},t_{j+1}]})<r_{0}< \frac{\alpha-\alpha'}{4}$.
By \textbf{(D)}, we have $d_{i}(x_{1},\tilde{\Phi}_{i}(\Phi_{i}(x_{1})))<4\epsilon_{i}$, $d_{i}(x_{2},\tilde{\Phi}_{i}(\Phi_{i}(x_{2})))<4\epsilon_{i}$, and
\begin{align}
d_{i}(\tilde{\Phi}_{i}(T(\gamma_{t_{j}})),\tilde{\Phi}_{i}(T(\gamma_{t_{j+1}}))) &< d_{\infty}(T(\gamma_{t_{j}}),T(\gamma_{t_{j+1}})) +4\epsilon_{i}\nonumber\\
&=\mathrm{Length}(\gamma|_{[t_{j},t_{j+1}]})+4\epsilon_{i}.\nonumber
\end{align}
By Lemma \ref{lem3.18}, for $i$ sufficiently large, the geodesic connecting $x_{1}$ and $\tilde{\Phi}_{i}(\Phi_{i}(x_{1}))$, the geodesic connecting $\tilde{\Phi}_{i}(T(\gamma_{t_{j}}))$ and $\tilde{\Phi}_{i}(T(\gamma_{t_{j+1}}))$ for each $j$, as well as the geodesic connecting $\tilde{\Phi}_{i}(\Phi_{i}(x_{2}))$ and $x_{2}$, are all contained in $\varphi_{i}^{-1}((a+\alpha',b-\alpha'))$.
We glue these geodesics to obtain a piecewise geodesic whose length is not larger than $\mathrm{Length}(\gamma)+4(C_{3}+3)\epsilon_{i}$.
Thus we obtain
\begin{align}\label{4.44441}
d_{i}^{\alpha,\alpha'}(x_{1},x_{2})\leq d_{Y}(S(\Phi_{i}(x_{1})),S(\Phi_{i}(x_{2})))+C_{4}\epsilon_{i},
\end{align}
where $C_{4}=4(C_{3}+3)$.

(2). For any $x_{1},x_{2}\in \varphi_{i}^{-1}((a+\alpha,b-\alpha))$, and for any $\delta>0$, let  $\eta:[0,1]\rightarrow \varphi_{i}^{-1}((a+\alpha',b-\alpha'))$ be an absolutely continuous curve connecting $x_{1}$ and $x_{2}$ such that $\mathrm{Length}(\eta)<d_{i}^{\alpha,\alpha'}(x_{1},x_{2})+\delta$.
We assume $\eta$ is parameterized by arc length.
By (\ref{4.44441}), we know $\mathrm{Length}(\eta)\leq C_{5}$ with $C_{5}$ only depends on $N,\bar{c},a,b,\alpha,\alpha'$.
Let $C_{6}=\lfloor\frac{4C_{5}}{\alpha'}\rfloor+1$, $s_{j}:=\frac{j}{C_{6}}$ for $j=0,1,\ldots,C_{6}$, then we have $d_{i}(\eta_{s_{j}},\eta_{s_{j+1}})\leq\mathrm{Length}(\eta|_{[s_{j},s_{j+1}]})<\frac{\alpha'}{4}$ for every $j$.
For $i$ sufficiently large, we have $\Phi_{i}(\eta_{s_{j}})\in\varphi_{\infty}^{-1}((a+\frac{\alpha'}{2},b-\frac{\alpha'}{2}))$ and
$d_{\infty} (\Phi_{i}(\eta_{s_{j}}),\Phi_{i}(\eta_{s_{j+1}}))<d_{i}(\eta_{s_{j}},\eta_{s_{j+1}})+\epsilon_{i}<\frac{\alpha'}{2}$
for every $j$.
By Lemma \ref{lem3.18}, a geodesic $\tilde{\eta}^{(j)}:[0,1]\rightarrow X_{\infty}$ connecting $\Phi_{i}(\eta_{s_{j}})$ and $\Phi_{i}(\eta_{s_{j+1}})$ must be contained in $\varphi_{\infty}^{-1}((a,b))$.
We glue these $\tilde{\eta}^{(j)}$ to get a piecewise geodesic $\tilde{\gamma}:[0,1]\rightarrow \varphi_{\infty}^{-1}((a,b))$ connecting $\Phi_{i}(x_{1})$ and $\Phi_{i}(x_{2})$.
By construction, $\mathrm{Length(\tilde{\gamma})}\leq \mathrm{Length(\eta)}+C_{6}\epsilon_{i}$.
Because $\Phi_{i}(x_{1}),\Phi_{i}(x_{2})\in \varphi_{\infty}^{-1}((a+\alpha,b-\alpha))$ and $Y$ has a product structure, there is always a geodesic $\gamma:[0,1]\rightarrow Z\times (a+\alpha,b-\alpha)$ connecting $S(\Phi_{i}(x_{1}))$ and $S(\Phi_{i}(x_{2}))$.
Thus $d_{Y}(S(\Phi_{i}(x_{1})),S(\Phi_{i}(x_{2})))=\mathrm{Length}(\gamma)\leq \mathrm{Length}(S(\tilde{\gamma}))=\mathrm{Length}(\tilde{\gamma})\leq \mathrm{Length(\eta)}+C_{6}\epsilon_{i}$.
By the arbitrariness of $\delta$, we have
\begin{align}\label{4.44447}
d_{Y}(S(\Phi_{i}(x_{1})),S(\Phi_{i}(x_{2})))\leq d_{i}^{\alpha,\alpha'}(x_{1},x_{2})+C_{6}\epsilon_{i}.
\end{align}
By (\ref{4.44441}) and (\ref{4.44447}), $S\circ\Phi_{i}$ is $\max\{C_{4},C_{6}\}\epsilon_{i}$-almost distance preserving.
The proof is completed.
\end{proof}

Note that Claim \ref{claim4.10} contradicts the assumption (4), hence we finish the proof of Theorem \ref{main-1.1}.
\end{proof}

\begin{prop}\label{prop4.3}
Suppose $(X,d)$ is a geodesic space, $E$ is a closed subset of $X$ such that $\mathrm{diam}(\partial E)<\infty$.
Let $\varphi(x)=d(x,E)$ be the distance function, $0<a<b$, $0<\alpha'<\alpha<\frac{b-a}{4}$ be real numbers.
Suppose there exists a length extended metric space $Z$ consisting of finitely many connected components $\{Z^{(j)}\}_{j=1}^{k}$, and there exist a number $\Psi$ with $0<\Psi<\sqrt{\Psi}<\frac{\alpha}{16}$, and a $\Psi$-Gromov-Hausdorff approximation $\Phi:(\varphi^{-1}((a+\alpha,b-\alpha)),d^{\alpha,\alpha'})\rightarrow (Z\times(a+\alpha,b-\alpha),d_{Z}\times d_{\mathrm{Eucl}})$ such that
\begin{description}
  \item[(a)] for any $x\in \varphi^{-1}((a+\alpha,b-\alpha))$, it holds
  \begin{align}\label{5.13}
  |\varphi(x)-r(\Phi(x))|<\Psi,
  \end{align}
  where $r:Z\times(a+\alpha,b-\alpha)\rightarrow(a+\alpha,b-\alpha)$ is the projection to the second factor;
  \item[(b)] for every $j$, $\mathrm{diam}(Z^{(j)})\leq C_{1}$ for some positive constant $C_{1}$, here the diameter is computed with respect to $d_{Z}$.
\end{description}

Then for any $s\in [a+2\alpha,b-2\alpha]$ and any connected component $V$ of $\varphi^{-1}((a+\alpha',b-\alpha'))$, any two points in $V\cap \varphi^{-1}(s)$ can be connected by a curve contained in $V\cap \varphi^{-1}((s-16\sqrt{\Psi},s+16\sqrt{\Psi}))$.

For any $\beta$ with $\beta\geq16\sqrt{\Psi}$ and $\beta< \min\{s-a-\alpha,b-\alpha-s\}$,
we define a distance $d_{s,\beta}^{V}$ on $V\cap \varphi^{-1}(s)$ by
\begin{align}
d_{s,\beta}^{V}(x_{1},x_{2})=\inf\{&\mathrm{Length(\tau)}\bigl|\tau:[0,1]\rightarrow V\cap\varphi^{-1}((s-\beta,s+\beta)), \\
&\tau_{0}=x_{1},\tau_{1}=x_{2}\}.\nonumber
\end{align}
Let $\pi:Z\times(a+\alpha,b-\alpha)\rightarrow Z$ be the projection onto the $Z$ factor, then $\pi(\Phi(V\cap \varphi^{-1}(s)))$ is contained in the same path-connected component of $Z$.
Assume this component is $Z^{(1)}$.
Then $\pi\circ\Phi:(V\cap \varphi^{-1}(s),d_{s,\beta}^{V})\rightarrow (Z^{(1)},d_{Z})$ is a $C\sqrt{\Psi}$-Gromov-Hausdorff approximation, where
$C$ is a positive constant depending on $C_{1}, \alpha$.
\end{prop}

\begin{proof}
Denote by $(Y,d_{Y})=(Z\times(a+\alpha,b-\alpha),d_{Z}\times d_{\mathrm{Eucl}})$ for simplicity.
Suppose $x_{1},x_{2}\in V\cap \varphi^{-1}(s)$, then $d^{\alpha,\alpha'}(x_{1},x_{2})<\infty$, and hence $d_{Y}(\Phi(x_{1}),\Phi(x_{2}))<\infty$.
Thus $\Phi(x_{1})$ and $\Phi(x_{2})$ lies in the same connected component of $Y$ which we assume it to be  $Z^{(1)}\times(a+\alpha,b-\alpha)$.
Without loss of generality, we assume $r(\Phi(x_{1}))\leq r(\Phi(x_{2}))$.
There always exists a geodesic $\gamma:[0,1]\rightarrow Z^{(1)}\times(a+\alpha,b-\alpha)$ connecting $\Phi(x_{1})$ and $\Phi(x_{2})$, such that $r(\gamma_{t})\in [r(\Phi(x_{1})), r(\Phi(x_{2}))]\subset(s-\Psi,s+\Psi)$ for every $t\in[0,1]$.
In addition,
\begin{align}
&d_{Z}(\pi(\Phi(x_{1})),\pi(\Phi(x_{2})))\leq\mathrm{Length}(\gamma)=d_{Y}(\Phi(x_{1}),\Phi(x_{2})) \\
\leq &d_{Z}(\pi(\Phi(x_{1})),\pi(\Phi(x_{2})))+r(\Phi(x_{2}))-r(\Phi(x_{1}))\leq C_{1}+2\Psi<C_{1}+2.\nonumber
\end{align}

Let $C_{2}:=\lfloor\frac{C_{1}+2}{\sqrt{\Psi}}\rfloor+1$, $t_{j}:=\frac{j}{C_{2}}$ for $j=0,1,\ldots,C_{2}$, then we have $\mathrm{Length}(\gamma|_{[t_{j},t_{j+1}]})<\sqrt{\Psi}< \frac{\alpha}{16}$.
Let $\tilde{\Phi}:(Y,d_{Y})\rightarrow(\varphi^{-1}((a+\alpha,b-\alpha)),d^{\alpha,\alpha'})$ be a $4\Psi$-Gromov-Hausdorff approximation which is an $\Psi$-inverse of $\Phi$,
then $d^{\alpha,\alpha'}(x_{1},\tilde{\Phi}(\Phi(x_{1})))<4\Psi$, $d^{\alpha,\alpha'}(x_{2},\tilde{\Phi}(\Phi(x_{2})))<4\Psi$, and
\begin{align}
d^{\alpha,\alpha'}(\tilde{\Phi}(\gamma_{t_{j}}),\tilde{\Phi}(\gamma_{t_{j+1}})) &< d_{Y}(\gamma_{t_{j}},\gamma_{t_{j+1}}) +4\Psi=\mathrm{Length}(\gamma|_{[t_{j},t_{j+1}]})+4\Psi<5\sqrt{\Psi}.\nonumber
\end{align}
In addition, by (\ref{5.13}) and $d_{Y}(\gamma_{t_{j}},\Phi(\tilde{\Phi}(\gamma_{t_{j}})))<\Psi$, we have
$\varphi(\tilde{\Phi}(\gamma_{t_{j}}))\in(r(\gamma_{t_{j}})-2\Psi,r(\gamma_{t_{j}})+2\Psi)\subset (s-3\Psi,s+3\Psi)$.
Then by Lemma \ref{lem3.18}, the geodesic connecting $x_{1}$ and $\tilde{\Phi}(\Phi(x_{1}))$, the geodesic connecting $\tilde{\Phi}(\gamma_{t_{j}})$ and $\tilde{\Phi}(\gamma_{t_{j+1}})$ for each $j$, as well as the geodesic connecting $\tilde{\Phi}(\Phi(x_{2}))$ and $x_{2}$, are all contained in $\varphi^{-1}((s-16\sqrt{\Psi},s+16\sqrt{\Psi}))$.
We glue these geodesics to obtain a piecewise geodesic
$\tau:[0,1]\rightarrow\varphi^{-1}((s-16\sqrt{\Psi},s+16\sqrt{\Psi}))$ connecting $x_{1}$ and $x_{2}$ with
$\mathrm{Length}(\tau)\leq\mathrm{Length}(\gamma)+C_{3}\sqrt{\Psi}$, where $C_{3}$ depends on $C_{1}, \alpha$.
In particular, we have
\begin{align}\label{5.19}
d_{s,16\sqrt{\Psi}}^{V}(x_{1},x_{2})\leq d_{Z}(\pi(\Phi(x_{1})),\pi(\Phi(x_{2})))+(C_{3}+2)\sqrt{\Psi}.
\end{align}
Hence, for any $\beta$ with $\beta\geq16\sqrt{\Psi}$ and $\beta< \min\{s-a-\alpha,b-\alpha-s\}$, we have
\begin{align}\label{5.18}
d_{s,\beta}^{V}(x_{1},x_{2})\leq d_{Z}(\pi(\Phi(x_{1})),\pi(\Phi(x_{2})))+(C_{3}+2)\sqrt{\Psi}.
\end{align}

On the other hand, for any $\delta>0$, let $\eta:[0,1]\rightarrow\varphi^{-1}((s-\beta,s+\beta))$ be a curve connecting $x_{1}$ and $x_{2}$ such that $\mathrm{Length}(\eta)\leq d_{s,\beta}^{V}(x_{1},x_{2})+\delta$, then
argue as in (2) of the proof of Claim \ref{claim4.10}, by making use of the product structure of $Y$, we can prove $d_{Y}(\Phi(x_{1}),\Phi(x_{2}))\leq \mathrm{Length}(\eta)+C_{4}\Psi$, where $C_{4}$ is a constant depending on $C_{1}, \alpha$.
Thus
\begin{align}\label{5.20}
d_{Z}(\pi(\Phi(x_{1})),\pi(\Phi(x_{2})))\leq d_{Y}(\Phi(x_{1}),\Phi(x_{2})) \leq d_{s,\beta}^{V}(x_{1},x_{2})+C_{4}\Psi+\delta.
\end{align}
By the arbitrariness of $\delta$ and (\ref{5.18}), we know the map $\pi\circ\Phi$ is $\max\{C_{3}+2,C_{4}\}\sqrt{\Psi}$-almost distance preserving.

For any $z\in Z^{(1)}$, by the assumptions on $\Phi$, there is an $x\in\varphi^{-1}((a+\alpha,b-\alpha))$ such that $d_{Y}((z,s+2\Psi),\Phi(x))<\Psi$, $s+\Psi<r(\Phi(x))<s+3\Psi$ and $s<\varphi(x)<s+4\Psi$.
It is easy to see that $x\in V$.
Let $y$ be the nearest point on $V\cap \varphi^{-1}(s)$ to $x$, then $d^{\alpha,\alpha'}(x,y)<4\Psi$.
Note that
\begin{align}
&d_{Z}(\pi(\Phi(y)),z)\leq d_{Z}(\pi(\Phi(y)),\pi(\Phi(x)))+d_{Z}(z,\pi(\Phi(x)))\\
<& d_{Y}(\Phi(y),\Phi(x))+d_{Y}((z,s+2\Psi),\Phi(x))< d^{\alpha,\alpha'}(y,x)+2\Psi<6\Psi.\nonumber
\end{align}
Thus $\pi\circ\Phi$ is $6\Psi$-almost onto.

In conclusion, $\pi\circ\Phi$ is a $C\sqrt{\Psi}$-Gromov-Hausdorff approximation.
\end{proof}

Proposition \ref{prop4.3} has the following corollary.

\begin{cor}\label{cor4.4}
Let $(X,d)$ be a geodesic space.
The close set $E$, the distance function $\varphi$ and real positive numubers $a,b,\alpha,\alpha'$ are the same as in Proposition \ref{prop4.3}.
Suppose there exist a length extended metric space $Z$, a real number $\Psi$ with $0<\Psi<\sqrt{\Psi}<\frac{\alpha}{16}$, and a $\Psi$-Gromov-Hausdorff approximation $\Phi:(\varphi^{-1}((a+\alpha,b-\alpha)),d^{\alpha,\alpha'})\rightarrow (Z\times(a+\alpha,b-\alpha),d_{Z}\times d_{\mathrm{Eucl}})$ satisfying the assumptions (a),(b) in Proposition \ref{prop4.3}.
Let $V$ be any connected component of $\varphi^{-1}((a+\alpha',b-\alpha'))$.
Then for any $s,\tilde{s}\in [a+2\alpha,b-2\alpha]$, any $\beta,\tilde{\beta}$ with $\beta,\tilde{\beta}\geq16\sqrt{\Psi}$, $\beta< \min\{s-a-\alpha,b-\alpha-s\}$ and $\tilde{\beta}< \min\{\tilde{s}-a-\alpha,b-\alpha-\tilde{s}\}$,
There is a $C\sqrt{\Psi}$-Gromov-Hausdorff approximation between $(V\cap\varphi^{-1}(s),d_{s,\beta}^{V})$ and $(V\cap\varphi^{-1}(\tilde{s}),d_{\tilde{s},\tilde{\beta}}^{V})$, where $C$ is a positive constant depending on $C_{1}, \alpha$.
In particular,
$$|\mathrm{diam}(V\cap\varphi^{-1}(s))-\mathrm{diam}'(V\cap\varphi^{-1}(\tilde{s}))|<C\sqrt{\Psi},$$
where $\mathrm{diam}$ and $\mathrm{diam}'$ are with respect to the distances $d_{s,\beta}^{V}$ and $d_{\tilde{s},\tilde{\beta}}^{V}$ respectively.
\end{cor}

\section{Proof of Theorem \ref{main-1.2} and some corollaries}\label{sec6}

\subsection{Busemann function}

Throughout this section, $(X, d, m)$ is always a noncompact $\mathrm{RCD}(0,N)$ space.
Given a geodesic ray $\gamma:[0,\infty)\rightarrow X$ emitted from $p$.
The function $b:X\rightarrow\mathbb{R}$ given by
$$b(x):=\lim_{t\rightarrow+\infty}(t-d(x,\gamma_{t}))$$
is called a Busemann function associated to $\gamma$.
$b$ is a $1$-Lipschitz function.

For any given $x\in X$, let $\eta^{t,x}:[0,d(x,\gamma_{t})]\rightarrow X$ be a unit speed geodesic connecting $x$ to $\gamma_{t}$, where $t\geq0$.
By the properness of $X$, there is a sequence $\{t_{n}\}$, with $t_{n}\rightarrow\infty$, such that $\eta^{t_{n},x}$ converge uniformly on compact sets to a geodesic ray $\eta^{(x)}:[0,\infty)\rightarrow X$ with $\eta^{(x)}_{0}=x$.
Such a ray $\eta^{(x)}$ is called a Busemann ray associated with $\gamma$.
One can prove that (see Lemma 3.1 in \cite{H16} for details), for every $t\geq0$, it holds
\begin{align}
b(\eta^{(x)}_{t})=b(x)+t.
\end{align}

Busemann function has close relation to optimal transport.
In fact, we can prove that, for any $t>0$, the function $-tb$ is $c$-concave, where $c=\frac{d^{2}}{2}$.
See Proposition 3.6 in \cite{H16} for more details.
Furthermore, making use of Theorem \ref{thm3.25}, we can prove:

\begin{prop}[see Proposition 3.13 of \cite{H16}]\label{3.30}
There is a Borel set $\tilde{\mathcal{T}}\subset X$, such that
\begin{enumerate}
  \item $m(X\setminus \tilde{\mathcal{T}})=0$;
  \item for any $x\in \tilde{\mathcal{T}}$, there exists exactly one Busemann ray $\eta^{(x)}:[0,\infty)\rightarrow X$. In addition, $\eta^{(x)}_{t}\in\tilde{\mathcal{T}}$ for every $t\geq0$.
\end{enumerate}
\end{prop}

Let $\tilde{R}=\{(x,y)\in X\times X\bigl||b(x)-b(y)|=d(x,y)\}$, $\tilde{R}(x)=\{y\bigl|(x,y)\in \tilde{R}\}$, then it is not hard to check that $\tilde{R}$ is an equivalent relation on $\tilde{\mathcal{T}}$,
and for all $x\in\tilde{\mathcal{T}}$, $\tilde{R}(x)\cap \tilde{\mathcal{T}}$ forms a single geodesic ray.

Let $K$ be a compact set.
Denote by
\begin{align}\label{symb-11}
\tilde{\Xi}(K):=\bigcup_{y\in K}\tilde{R}(y),
\end{align}
\begin{align}\label{symb-12}
\tilde{\Xi}_{[s,t]}(K):=\tilde{\Xi}(K)\cap b^{-1}([s,t]),
\end{align}
\begin{align}\label{symb-13}
\tilde{\Xi}_{s}(K):=\tilde{\Xi}(K)\cap b^{-1}(s),
\end{align}

We can prove the following:

\begin{prop}[Proposition 3.21 in \cite{H16}]\label{prop3.19}
Suppose $(X,d,m)$ is a noncompact $\mathrm{RCD}(0,N)$ space.
Let $K\subset b^{-1}((-\infty,r_{0}])$ be a compact set, then
\begin{align}\label{3.17-2}
\frac{m(\tilde{\Xi}_{[r_{1},r_{2}]}(K))}{r_{2}-r_{1}}\leq\frac{m(\tilde{\Xi}_{[r_{2},r_{3}]}(K))}{r_{3}-r_{2}}
\end{align}
holds for any $r_{3}> r_{2}>r_{1}\geq r_{0}$.
\end{prop}

\subsection{Noncompact \textmd{RCD}(0,N) spaces with linear volume growth}

Recall that a noncompact $\textmd{RCD}(0,N)$ space has at least linear volume growth, see Proposition 2.8 of \cite{H16} for a proof.
It is then a natural problem to investigate noncompact $\mathrm{RCD}(0,N)$ spaces $(X,d,m)$ with linear volume growth, i.e. there is some positive constant $V_{0}$ such that
\begin{align}\label{min-vol}
\limsup_{R\rightarrow\infty}\frac{m(B_{p}(R))}{R}=V_{0}.
\end{align}

In this section, the metric measure space $(X,d,m)$ is assumed to be a noncompact $\mathrm{RCD}(0,N)$ space satisfying (\ref{min-vol}).
On such metric measure spaces, making use of Proposition \ref{prop3.19} and the Bishop-Gromov volume comparison property, we can prove the following theorem.

\begin{thm}[Theorem 1.2 in \cite{H16}]\label{compact-level-set}
Suppose $(X,d,m)$ is a noncompact $\mathrm{RCD}(0,N)$ space satisfying (\ref{min-vol}), and $b$ is the Busemann function associated to a geodesic ray $\gamma$.
Then we have
\begin{align}\label{4.7}
\limsup_{r\rightarrow+\infty}\frac{\mathrm{diam}(b^{-1}(r))}{r}\leq C_{0}\leq2,
\end{align}
where the diameter of $b^{-1}(r)$ is computed with respect to the distance $d$.
In particular, for any $r$, $b^{-1}(r)$ is compact.
\end{thm}

Then by (\ref{3.17-2}), it is easy to see that for any $r_{3}> r_{2}>r_{1}$, it holds
\begin{align}\label{3.17-5}
\frac{m(b^{-1}([r_{1},r_{2}]))}{r_{2}-r_{1}}\leq\frac{m(b^{-1}([r_{2},r_{3}]))}{r_{3}-r_{2}}.
\end{align}
Hence for every $s$,
\begin{align}
m_{-1}(b^{-1}(s)):=\lim_{t\downarrow s}\frac{m(b^{-1}([s,t]))}{t-s}
\end{align}
is well-defined, and
\begin{align}\label{3.17-1}
m_{-1}(b^{-1}(r_{1}))\leq\frac{m(b^{-1}([r_{1},r_{2}])}{r_{2}-r_{1}}\leq m_{-1}(b^{-1}(r_{2}))
\end{align}
holds for any $r_{2}>r_{1}$.
Furthermore, by Proposition 4.1 of \cite{H16}, for any $r_{2}>r_{1}$, we have
\begin{align}\label{4.2}
m_{-1}(b^{-1}(r_{1}))\leq m_{-1}(b^{-1}(r_{2}))\leq V_{0}.
\end{align}
In particular, the limit
\begin{align}
V_{\infty}:=\lim_{r\rightarrow+\infty}m_{-1}(b^{-1}(r))
\end{align}
exists and $0<V_{\infty}\leq V_{0}$.

By (\ref{3.17-1}) and (\ref{4.2}), it is easy to prove that for any $\delta\in(0,1)$, there exists $R_{\delta}$ such that for every $r_{3}> r_{2}>r_{1}\geq R_{\delta}$, we have
\begin{align}\label{4.22228}
\frac{m(b^{-1}([r_{2},r_{3}]))}{r_{3}-r_{2}}\leq (1+\delta)\frac{m(b^{-1}([r_{1},r_{3}]))}{r_{3}-r_{1}}.
\end{align}

Following \cite{Sor99}, we define the almost intrinsic diameter of level sets of a Busemann function.

\begin{defn}\label{def6.4}
Given $R>0$, $r\in(0,R)$, the $r$-almost intrinsic diameter of $b^{-1}(R)$, denoted by $\mathrm{diam}_{r}(b^{-1}(R))$, is defined to be
\begin{align}
\mathrm{diam}_{r}(b^{-1}(R)):=\max&\{\mathrm{diam}^{V}(V\cap b^{-1}(R))\bigl| V \text{ is a connected} \nonumber\\
&\text{component of }b^{-1}(R-r,R+r)\},\nonumber
\end{align}
where
\begin{align}
\mathrm{diam}^{V}(V\cap b^{-1}(R)):=\sup_{x,y\in V\cap b^{-1}(R)}&\{\text{the infimum of the length of curves} \nonumber\\
&\text{ connecting }x, y \text{ and contained in }V \}.\nonumber
\end{align}
\end{defn}

Similar to \cite{Sor99}, we will prove the following theorem.

\begin{thm}\label{thm6.1}
Given any $\zeta\in(0, \frac{1}{3})$, we have
\begin{align}
\lim_{R\rightarrow\infty}\frac{\mathrm{diam}_{\zeta R}(b^{-1}(R))}{R}= 0.
\end{align}
\end{thm}

The key to prove Theorem \ref{thm6.1} is the almost rigidity theorem \ref{main-1.1}.
First of all, we need the following lemma:

\begin{lem}\label{lem6.1}
For any $s$, let $E_{s}:=b^{-1}([s,\infty))$, then
\begin{align}\label{4.22238}
d(x,E_{s})=\max\{s-b(x),0\}.
\end{align}
In addition, for any $x\in b^{-1}((-\infty,s))$, let $\eta^{(x)}$ be a Busemann ray, then any $y\in\eta^{(x)}\cap b^{-1}(s)$ satisfies $d(x,y)=d(x,E_{s})=s-b(x)$.
\end{lem}
The proof of Lemma \ref{lem6.1} can be found in Section 3.1 of \cite{H16}.

\begin{proof}[Proof of Theroem \ref{thm6.1}]
For any $s_{1}$, we defined a new metric measure structure $(X,\bar{d},\bar{m})$ such that $\bar{d}=\frac{1}{s_{1}}d$, and $\bar{m}=m$.
For simplicity, we denote by $E:=E_{s_{1}}$ and $\varphi(x)=\bar{d}(x,E)$.

Note that by (\ref{4.7}), there is a constant $R_{1}$ such that for any $s_{1}\geq R_{1}$, $\partial E=b^{-1}(s_{1})$ satisfies
\begin{align}\label{5.4}
\mathrm{diam}(\partial E)\leq 3,
\end{align}
here the diameter is computed with respect to $\bar{d}$.

We take $a=\frac{1}{10}$, $b=\frac{9}{10}$, $c=\frac{1}{2}$.
By Lemma \ref{lem6.1} and (\ref{4.22228}), it is easy to see that for any given $\delta\in(0,1)$, there exists a positive constant $R_{\delta}$ such that for any $s_{1}\geq R_{\delta}$, we have
\begin{align}\label{4.22229}
\frac{\bar{m}(\varphi^{-1}([a,c]))}{\bar{m}(\varphi^{-1}([a,b]))}\leq(1+\delta)\frac{c-a}{b-a}.
\end{align}

In addition, by Proposition \ref{prop3.19} and Lemma \ref{lem6.1} , it is easy to see that $(X,\bar{d},\bar{m})$ satisfies $\mathrm{MDADF}$ property on $\varphi^{-1}((a,b))$.

We take $\alpha=\frac{1}{100}$, $\alpha'=\frac{1}{1000}$.
Let $\Psi=\Psi(\delta|a,b,c,\alpha,\alpha',3,N)$ be the function given in Theorem \ref{main-1.1}.
For any $s_{1}$ sufficiently large such that (\ref{5.4}) and (\ref{4.22229}) hold, by Theorem \ref{main-1.1}, there exist a length extended metric space $(Z,d_{Z})$, and a $\Psi$-Gromov-Hausdorff approximation $\Phi:(\varphi^{-1}((a+\alpha,b-\alpha)),\bar{d}^{\alpha,\alpha'})\rightarrow (Z\times(a+\alpha,b-\alpha),d_{Z}\times d_{\mathrm{Eucl}})$ such that
\begin{description}
  \item[(a)] for any $x\in \varphi^{-1}((a+\alpha,b-\alpha))$, it holds
  \begin{align}
  |\varphi(x)-r(\Phi(x))|<\Psi;
  \end{align}
  \item[(b)] $Z$ consists of at most $C_{0}$ connected components $\{Z^{(j)}\}_{j=1}^{k}$, where $C_{0}$ is a constant;
  \item[(c)] for every $j$, the diameter of $Z^{(j)}$ is at most $C_{1}$.
\end{description}
Here $C_{0},C_{1}$ are constants independent of $\delta$.

Now we choose $\delta$ sufficiently small such that $\Psi<\sqrt{\Psi}<\frac{\alpha}{16}$.
For any connected component $V$ of $\varphi^{-1}((a+\alpha',b-\alpha'))$, and for any $s\in[\frac{2}{5},\frac{3}{5}]$, and any $\zeta\in(40\sqrt{\Psi},\frac{1}{3})$, we consider the distance $\bar{d}^{V}_{s,\zeta s}$ on $V\cap\varphi^{-1}(s)$ as in Proposition \ref{prop4.3}.
Then by Corollary \ref{cor4.4}, for any $s,\tilde{s}\in [\frac{2}{5},\frac{3}{5}]$, we have
\begin{align}\label{6.4004}
|\mathrm{diam}^{V}_{s,\zeta s}(V\cap\varphi^{-1}(s))-\mathrm{diam}^{V}_{\tilde{s},\zeta \tilde{s} }(V\cap\varphi^{-1}(\tilde{s}))|<C_{2}\sqrt{\Psi},
\end{align}
where $\mathrm{diam}^{V}_{s,\zeta s}$ and $\mathrm{diam}^{V}_{\tilde{s},\zeta \tilde{s}}$ are computed with respect to the distances $\bar{d}^{V}_{s,\zeta s}$ and $\bar{d}^{V}_{\tilde{s},\zeta\tilde{s}}$ respectively, and $C_{2}$ is a positive constant independent of $\delta$.

By Definition \ref{def6.4}, Lemma \ref{lem6.1}, Proposition \ref{prop4.3} and the above construction, it is easy to see that for any $s_{1}$ sufficiently large and for any $s\in [\frac{2}{5},\frac{3}{5}]$, any $\zeta\in(40\sqrt{\Psi},\frac{1}{3})$, it holds
\begin{align}
\frac{1}{s_{1}}\mathrm{diam}_{\zeta ss_{1}}(b^{-1}(ss_{1}))=&\max\{\mathrm{diam}^{V}_{s,\zeta s}(V\cap\varphi^{-1}(s))\bigl| V \text{ is a connected } \\
&\text{ component of }\varphi^{-1}((a+\alpha',b-\alpha'))\}.\nonumber
\end{align}
Hence for any $s_{1}$ sufficiently large and any $s, \tilde{s}\in [\frac{2}{5},\frac{3}{5}]$, any $\zeta\in(40\sqrt{\Psi},\frac{1}{3})$, we have
\begin{align}\label{6.21}
|\mathrm{diam}_{\zeta ss_{1}}(b^{-1}(ss_{1}))-\mathrm{diam}_{\zeta \tilde{s}s_{1}}(b^{-1}(\tilde{s}s_{1}))| <C_{2}s_{1}\sqrt{\Psi}.
\end{align}

Now we fix $\zeta\in(0, \frac{1}{3})$, and then choose $\delta$ sufficiently small such that $\Psi<\sqrt{\Psi}<\min\{\frac{\alpha}{16},\frac{\zeta}{40}\}$.
By (\ref{6.21}), there is $R_{0}$ depending on $\delta$ such that for any $R\geq R_{0}$, and any $r\in[R,\frac{3}{2}R]$, we have
\begin{align}\label{6.22}
|\mathrm{diam}_{\zeta R}(b^{-1}(R))-\mathrm{diam}_{\zeta r}(b^{-1}(r))| <C_{3}\sqrt{\Psi} R,
\end{align}
where $C_{3}=\frac{5}{2}C_{2}$ is independent of $\delta$.

For any $k\in \mathbb{Z}^{+}$, by (\ref{6.22}), we have
\begin{align}\label{6.23}
&\mathrm{diam}_{\zeta \bigl(\frac{3}{2}\bigr)^{k}R_{0}}(b^{-1}(\bigl(\frac{3}{2}\bigr)^{k}R_{0}))\\
\leq &\mathrm{diam}_{\zeta \bigl(\frac{3}{2}\bigr)^{k-1}R_{0}}(b^{-1}(\bigl(\frac{3}{2}\bigr)^{k-1}R_{0})) +C_{3}\sqrt{\Psi} \bigl(\frac{3}{2}\bigr)^{k-1}R_{0}\nonumber\\
\leq &\mathrm{diam}_{\zeta \bigl(\frac{3}{2}\bigr)^{k-2}R_{0}}(b^{-1}(\bigl(\frac{3}{2}\bigr)^{k-2}R_{0})) +C_{3}\sqrt{\Psi} \biggl[\bigl(\frac{3}{2}\bigr)^{k-1}+\bigl(\frac{3}{2}\bigr)^{k-2}\biggr]R_{0}\nonumber\\
\leq &\ldots\nonumber\\
\leq &\mathrm{diam}_{\zeta R_{0}}(b^{-1}(R_{0})) +C_{3}\sqrt{\Psi} \biggl[\bigl(\frac{3}{2}\bigr)^{k-1}+\bigl(\frac{3}{2}\bigr)^{k-2}+\ldots+1\biggr]R_{0} \nonumber\\
\leq &\mathrm{diam}_{\zeta R_{0}}(b^{-1}(R_{0})) +C_{4}\sqrt{\Psi} \bigl(\frac{3}{2}\bigr)^{k}R_{0}, \nonumber
\end{align}
where $C_{4}$ is a positive constant independent of $\delta$.

For every $R >R_{0}$, there exists $k\in \mathbb{Z}^{+}$ such that $R\in[\bigl(\frac{3}{2}\bigr)^{k}R_{0}, \bigl(\frac{3}{2}\bigr)^{k+1}R_{0}]$.
By (\ref{6.22}) and (\ref{6.23}), we have
\begin{align}
&\frac{\mathrm{diam}_{\zeta R}(b^{-1}(R))}{R}\leq \frac{1}{R}\biggl[\mathrm{diam}_{\zeta \bigl(\frac{3}{2}\bigr)^{k}R_{0}}(b^{-1}(\bigl(\frac{3}{2}\bigr)^{k}R_{0}))+C_{3}\sqrt{\Psi} \bigl(\frac{3}{2}\bigr)^{k}R_{0} \biggr]\\
\leq&\frac{1}{R}\biggl[\mathrm{diam}_{\zeta R_{0}}(b^{-1}(R_{0})) +(C_{3}+C_{4})\sqrt{\Psi} \bigl(\frac{3}{2}\bigr)^{k}R_{0}\biggr]\nonumber\\
\leq&\frac{\mathrm{diam}_{\zeta R_{0}}(b^{-1}(R_{0}))}{R} +(C_{3}+C_{4})\sqrt{\Psi} .\nonumber
\end{align}
Let $R\rightarrow\infty$, we have
\begin{align}
\limsup_{R\rightarrow\infty}\frac{\mathrm{diam}_{\zeta R}(b^{-1}(R))}{R}\leq (C_{3}+C_{4})\sqrt{\Psi}.
\end{align}
Finally, let $\delta\rightarrow0$, and hence $\Psi\rightarrow0$, we obtain
\begin{align}
\lim_{R\rightarrow\infty}\frac{\mathrm{diam}_{\zeta R}(b^{-1}(R))}{R}=0.
\end{align}
The proof is completed.
\end{proof}

Now we prove Theorems \ref{main-1.2} and \ref{main-1.3}.

\begin{proof}[Proof of Theorem \ref{main-1.3}]
Suppose the conclusion does not hold, then there exists a sequence of positive number $\{r_{i}\}$ with $r_{i}\rightarrow+\infty$ and
\begin{align}\label{6.001}
\frac{\mathrm{diam}(b^{-1}(r_{i}))}{r_{i}}\geq C>0
\end{align}
for every $i$.
Thus there exist $x_{i}, y_{i}\in b^{-1}(r_{i})$, and a unit speed geodesic $\sigma^{(i)}:[0,L_{i}]\rightarrow X$ connecting $x_{i}$ and $y_{i}$, where $L_{i}:=\mathrm{Length}(\sigma^{(i)})=h_{i}r_{i}$ satisfies $C\leq h_{i}\leq 2$.

Suppose there is a subsequence of $\{r_{i}\}$, still denote by $\{r_{i}\}$, such that
$$\sigma^{(i)}\subset b^{-1}([\frac{3r_{i}}{4},\frac{5r_{i}}{4}]).$$
In this case, we have
$$\mathrm{diam}_{\frac{1}{4}r_{i}}(b^{-1}(r_{i}))\geq h_{i}r_{i}\geq Cr_{i},$$
which contradicts Theorem \ref{thm6.1}.

Thus we may assume that $\sigma^{(i)}$ is not a subset of $b^{-1}([\frac{3r_{i}}{4},\frac{5r_{i}}{4}])$ for any $i$.

Let
\begin{align}
S_{i}:=\min_{t\in[0,L_{i}]}b(\sigma^{(i)}_{t})
\end{align}
and
\begin{align}
T_{i}:=\max_{t\in[0,L_{i}]}b(\sigma^{(i)}_{t}).
\end{align}

Suppose there is a subsequence of $\{r_{i}\}$, still denote by $\{r_{i}\}$, such that $T_{i}>\frac{5r_{i}}{4}$.
Then it is easy to see that there are points $x_{i}',y_{i}'\in \mathrm{Im}(\sigma^{(i)})\cap b^{-1}(\frac{4T_{i}}{5})$ such the the part of $\sigma^{(i)}$ between $x_{i}',y_{i}'$ pass through a point in  $b^{-1}(T_{i})$.
Thus $d(x_{i}',y_{i}')\geq\frac{2T_{i}}{5}$, and
$$\mathrm{diam}_{\frac{1}{4}(\frac{4T_{i}}{5})}(b^{-1}(\frac{4T_{i}}{5}))\geq\frac{2T_{i}}{5}.$$
This contradicts Theorem \ref{thm6.1} because $\frac{4T_{i}}{5}\rightarrow\infty$.

Thus we may assume $S_{i}<\frac{3r_{i}}{4}$ for every $i$.

Suppose $S_{i}\rightarrow\infty$.
Then it is easy to see that there are points $x_{i}',y_{i}'\in \mathrm{Im}(\sigma^{(i)})\cap b^{-1}(\frac{4S_{i}}{3})$ such the the part of $\sigma^{(i)}$ between $x_{i}',y_{i}'$ pass through a point in  $b^{-1}(S_{i})$.
Thus $d(x_{i}',y_{i}')\geq\frac{2S_{i}}{3}$ and
$$\mathrm{diam}_{\frac{1}{4}(\frac{4S_{i}}{3})}(b^{-1}(\frac{4S_{i}}{3}))\geq \frac{2S_{i}}{3}.$$
This contradicts Theorem \ref{thm6.1}.

Thus we may assume there is a subsequence of $i$, still denote by $i$, such that $S_{i}\leq R$ for some $R$.
Thus every $\sigma^{(i)}$ passes through the compact set $b^{-1}(R)$.
Since $d(\sigma^{(i)}_{0},b^{-1}(R))\rightarrow\infty$ and $d(\sigma^{(i)}_{1},b^{-1}(R))\rightarrow\infty$,
after reparameterizing and extract a suitable subsequence of $\sigma^{(i)}$, this subsequence will converge to a line $\sigma^{(\infty)}$ passing through $b^{-1}(R)$.
By Gigli's splitting theorem for $\mathrm{RCD}(0,N)$ space, $(X,d,m)$ splits.
It is easy to see that in this case the Busemann function associated with $\sigma^{(\infty)}$ shares the same level sets of $b$, and $b^{-1}(r)$ is totally geodesic for every $r$.
Thus $S_{i}<\frac{3r_{i}}{4}$ cannot happen.

In conclusion, (\ref{6.001}) does not hold.
This completes the proof of Theorem \ref{main-1.3}.
\end{proof}

\begin{proof}[Proof of Theorem \ref{main-1.2}]
Suppose $(X,d,m)$ is a noncompact $\mathrm{RCD}(0,N)$ space with linear volume growth and $(X,d,m)$ does not split.
For the fixed point $p$, let $\gamma:[0,\infty)\rightarrow X$ be a geodesic ray such that $\gamma_{0}=p$.
Let $b$ be the Busemann function associated with $\gamma$, then (\ref{1.3}) holds.
It is easy to see that, for any $r>0$, $d(p,\gamma_{r})=d(p,b^{-1}([r,\infty)))=r$, and $\partial B_{p}(r)\subset b^{-1}((-\infty,r])$.

Let $S_{r}:=\min_{x\in \partial B_{p}(r)}b(x)$, we claim that
\begin{align}\label{6.4002}
\lim_{r\rightarrow\infty}\frac{S_{r}}{r}=1.
\end{align}

By the triangle inequality, it is easy to see
\begin{align}\label{6.4001}
\mathrm{diam}(\partial B_{p}(r))\leq \mathrm{diam}(b^{-1}(r))+2(r-S_{r}).
\end{align}
If the claim is true, we divide both sides of (\ref{6.4001}) by $r$, and let $r\rightarrow\infty$, then (\ref{1.2}) follows from (\ref{1.3}) and (\ref{6.4002}).

Suppose the claim does not hold, then there exist $\delta\in(0,1)$ and a sequence of positive numbers $\{r_{i}\}$, a sequence of points $\{x_{i}\}$, such that $r_{i}\rightarrow\infty$, $d(x_{i},p)=r_{i}$, $b(x_{i})=S_{r_{i}}< (1-\delta)r_{i}$.

Suppose there is a subsequence of $i$, still denoted by $i$ for simplicity, such that $S_{r_{i}}\rightarrow\infty$, then by the triangle inequality, we have
\begin{align}\label{6.4003}
r_{i}=d(x_{i},p)\leq S_{r_{i}}+d(\gamma_{S_{r_{i}}},x_{i})\leq S_{r_{i}}+\mathrm{diam}(b^{-1}(S_{r_{i}})).
\end{align}
Divide both sides of (\ref{6.4003}) by $S_{r_{i}}$, we have
\begin{align}
\frac{1}{1-\delta}<\frac{r_{i}}{S_{r_{i}}}<1+\frac{\mathrm{diam}(b^{-1}(S_{r_{i}}))}{S_{r_{i}}},
\end{align}
Let $i\rightarrow\infty$, then by (\ref{1.3}), we get a contradiction.

Thus we may assume $S_{r_{i}}\leq C_{1}$ for some constant $C_{1}$.

If in addition that there is another constant $C_{2}$ such that $C_{2}\leq S_{r_{i}}\leq C_{1}$ for all $i$, then by Theorem \ref{main-1.3}, $\{x_{i}\}$ is a bounded subset, this contradicts the assumption that $r_{i}=d(p,x_{i})\rightarrow\infty$.

Hence we may assume $S_{r_{i}}\rightarrow -\infty$ as $i\rightarrow\infty$.
Let $\sigma^{(i)}$ be a unit speed geodesic connecting $x_{i}$ and $\gamma_{r_{i}}$.
Because every $\sigma^{(i)}$ pass through the compact set $b^{-1}(0)$, by Arzela-Ascoli Theorem, after reparameterize and extract a suitable subsequence of $\sigma^{(i)}$, the new subsequence will converge to a line $\sigma^{(\infty)}$.
Then by Gigli's splitting theorem, $(X,d,m)$ splits.
This contradicts the assumption in Theorem \ref{main-1.2} and the proof of the claim is completed.
\end{proof}

\begin{defn}\label{tangcone}
Suppose $(X,d,m)$ is a noncompact $\mathrm{RCD}(0,N)$ space, $p\in X$ a fixed point.
Suppose $r_{i}\rightarrow\infty$, and $(X_{i},p_{i},d_{i},m_{i})$ is a sequence of pointed metric measure spaces such that $X_{i}=X$, $p_{i}=p$, $d_{i}=\frac{1}{r_{i}}d$, $m_{i}=\frac{1}{m(B_{p}(r_{i}))}m$.
Then up to a subsequence, the $\mathrm{RCD}(0,N)$ spaces $(X_{i},p_{i},d_{i},m_{i})$ will converge in the pointed measured Gromov-Hausdorff distance to a $\mathrm{RCD}(0,N)$ space $(X_{\infty},p_{\infty},d_{\infty},m_{\infty})$ with $m_{\infty}(B_{p_{\infty}}(1))=1$.
We call $(X_{\infty},p_{\infty},d_{\infty},m_{\infty})$ a tangent cone at infinity of $(X,d,m)$.
In general, a tangent cone at infinity may depend on the choice of $\{r_{i}\}$.
\end{defn}

\begin{proof}[Proof of Proposition \ref{main-1.4}]
The case when $(X,d,m)$ splits is obvious, so in the following we always assume $(X,d,m)$ does not split.

Suppose $(X_{\infty},p_{\infty},d_{\infty},m_{\infty})$ is a tangent cone at infinity given by $(X_{i},p_{i},d_{i},m_{i})\xrightarrow{pmGH}(X_{\infty},p_{\infty},d_{\infty},m_{\infty})$ as in Definition \ref{tangcone}.
We claim that, for any $r>0$, $\partial B_{p_{\infty}}(r)$ consists of exactly one point.

If the claim dose not hold, then for some $r>0$, there exist $x_{\infty},y_{\infty}\in X_{\infty}$ such that $d_{\infty}(x_{\infty},p_{\infty})=d_{\infty}(y_{\infty},p_{\infty})=r$ and $d_{\infty}(x_{\infty},y_{\infty})=\epsilon>0$.
Let $x_{i},y_{i}\in X_{i}$ such that $x_{i}\xrightarrow{GH} x_{\infty}$, $y_{i}\xrightarrow{GH} y_{\infty}$, then for $i$ sufficiently large, we have $d_{i}(x_{i},y_{i})>\frac{\epsilon}{2}$ and
$|d_{i}(x_{i},p_{i})-r|<\delta$, $|d_{i}(y_{i},p_{i})-r|<\delta$ for some $0<\delta<\min\{\frac{\epsilon}{16},r\}$.
By the triangle inequality, it is easy to see $d_{i}(x_{i},y_{i})<4\delta+\mathrm{diam}_{i}(\partial B_{p_{i}}(r-\delta))$, where $\mathrm{diam}_{i}$ is computed with respect to $d_{i}$.
Thus
$$\frac{\epsilon}{2}<4\delta+\frac{1}{r_{i}}\mathrm{diam}(\partial B_{p}((r-\delta)r_{i})).$$
Let $i\rightarrow\infty$, then by (\ref{1.2}), we obtain $\frac{\epsilon}{2}<4\delta$, which is a contradiction.

By the claim, $(X_{\infty},d_{\infty})$ is isometric to $([0,\infty),d_{Eucl})$.
Since $(X_{\infty},d_{\infty},m_{\infty})$ is a $\mathrm{RCD}(0,N)$ space, we can prove that
$m$ is absolutely continuous with respect to $\mathcal{L}^{1}$, that the density function  $h:[0,\infty)\rightarrow(0,\infty)$ has a locally Lipschitz representative, and satisfies estimates (\ref{3.14}) and (\ref{3.15}).
See the proof of Theorem \ref{thm3.15} for similar arguments.
\end{proof}

\section{Proof of Theorem \ref{main-1.5}}\label{sec7}

Suppose $(X,d,m)$ is a noncompact $\mathrm{RCD}(0,N)$ space.
We fix a point $p\in X$, and denote by $\rho(x)=d(x,p)$.

By \cite{J14} and \cite{AGS14}, a harmonic function $u$ always has a locally Lipschitz representative, thus in the following, we will always assume $u$ is locally Lipschitz.

For any $k\geq0$, let
\begin{align}
\mathcal{H}_{k}(X)=\{u&\in W^{1,2}_{\mathrm{loc}}(X)\cap\mathrm{Lip}_{\mathrm{loc}}(X)\mid \mathbf{\Delta} u=0, u(p)=0, \nonumber\\
&|u(x)|\leq C(\rho(x)^{k}+1) \text{ for some }C\}.\nonumber
\end{align}

Recall that the gradient estimate can be generalized to $\mathrm{RCD}$ spaces, see e.g. Theorem 1.6 of \cite{ZZ16}.
By the same argument as on manifolds, one can easily prove that, for $k<1$, $\mathcal{H}_{k}(X)$ only contains the zero function.
Thus in the following we always assume $k\geq1$.
Also by the gradient estimate, for any $u\in \mathcal{H}_{k}(X)$, we have
\begin{align}\label{4.11111}
|\mathrm{lip} u|(x)\leq C(\rho(x)^{k-1}+1),
\end{align}
where $C$ is some positive constant depending on $u$.

\begin{proof}[Proof of Theorem \ref{main-1.5}]
We argue by contradiction.
Suppose $(X,d,m)$ does not split, but there exists a non-trivial function $u\in \mathcal{H}_{k}(X)$ for some $k\geq1$.

Define a function $f:[0,\infty)\rightarrow[0,\infty)$ by
\begin{align}\label{7.002}
f(s)=\int_{B_{p}(s)}|D u|^{2}dm.
\end{align}
Obviously $f$ is a nondecreasing function, and $f$ does not vanish identically.
Furthermore, by (\ref{4.11111}) and the volume comparison property, we have
\begin{align}
f(s)\leq C(1+s^{2k+N-2}).
\end{align}

The following lemma is a special case of Lemma 3.1 in \cite{CM97a}:
\begin{lem}\label{4.11112}
Suppose $f$ is a nonnegative nondecreasing functions defined on $(0,\infty)$, $f$ does not vanish identically, and there are $d,K>0$ such that $f(r)\leq K(r^{d}+1)$.
Then for every $\Omega>1$, and any $C>\Omega^{d}$, there exist infinitely many integers $m$ such that
$$f(\Omega^{m+1})\leq Cf(\Omega^{m}).$$
\end{lem}

Apply Lemma \ref{4.11112} to the $f$ constructed in (\ref{7.002}), there exists a sequence of positive integers $\{m_{i}\}$ with $m_{i}\rightarrow\infty$, such that
\begin{align}\label{5.1111}
f(2^{m_{i}+1})\leq 2^{2k+N-1}f(2^{m_{i}}).
\end{align}

Take $r_{i}=2^{m_{i}}$.
As in Definition \ref{tangcone}, we obtain a sequence of $\mathrm{RCD}(0,N)$ spaces $(X_{i},p_{i},d_{i},m_{i})$, and assume $(X_{i},p_{i},d_{i},m_{i})\xrightarrow{pmGH}(X_{\infty},p_{\infty},d_{\infty},m_{\infty})$.
Denote by $B_{p_{i}}^{(i)}(r)=\{x\in X_{i}\mid d_{i}(x,p_{i})< r\}$, $B_{p_{\infty}}^{(\infty)}(r)=\{x\in X_{\infty}\mid d_{\infty}(x,p_{\infty})< r\}$.

Define
\begin{align}\label{ui}
u^{(i)}(x)=\frac{u(x)}{{r_{i}}\bigl(\frac{f(r_{i})}{m(B_{p}(r_{i}))}\bigr)^{\frac{1}{2}}},
\end{align}
then $u^{(i)}$ is harmonic on $(X_{i},d_{i},m_{i})$ and
\begin{align}\label{7.003}
\int_{B^{(i)}_{p_{i}}(1)}|D^{(i)} u^{(i)}|^{2}dm_{i}=1,
\end{align}
where $|D^{(i)} u^{(i)}|$ denote the minimal weak upper gradient of $u^{(i)}$ with respect to $(X_{i},d_{i},m_{i})$.

Furthermore, by (\ref{5.1111}), we have
\begin{align}\label{7.7}
\int_{B^{(i)}_{p_{i}}(2)}|D^{(i)} u^{(i)}|^{2}dm_{i}\leq 2^{2k+N-1} \int_{B^{(i)}_{p_{i}}(1)}|D^{(i)} u^{(i)}|^{2}dm_{i}=2^{2k+N-1},
\end{align}

Recall that in \cite{EKS15}, it is proved that on $\mathrm{RCD}$ spaces, a global version of Bochner formula holds.
In addition, on $\mathrm{RCD}$ spaces there exist good cut-off functions, see e.g. Lemma 3.1 of \cite{MN14}.
Then one can prove a localized Bochner formula, see Theorem 3.5 of \cite{ZZ16} for details.
In particular, since $u^{(i)}$ is harmonic on $B_{p_{i}}(2)$, we have $|D^{(i)} u^{(i)}|^{2}\in W^{1,2}(B_{p_{i}}(\frac{15}{8}))\cap L^{\infty}(B_{p_{i}}(\frac{15}{8}))$ and
\begin{align}
\mathbf{\Delta}(|D^{(i)} u^{(i)}|^{2})\geq 0 \quad\text{on}\quad B_{p_{i}}(\frac{15}{8}).
\end{align}
Then we apply the weak Harnack inequality (see e.g. Theorem 8.4 in \cite{BB11}) to $|D^{(i)} u^{(i)}|^{2}$ and obtain
\begin{align}
\||D^{(i)}u^{(i)}|^{2}\|_{L^{\infty}(B^{(i)}_{p_{i}}(\frac{3}{2}))} \leq \frac{C(N)}{m_{i}(B^{(i)}_{p_{i}}(\frac{15}{8}))}\int_{B^{(i)}_{p_{i}}(\frac{15}{8})}|D^{(i)} u^{(i)}|^{2}dm_{i}.
\end{align}
Hence by (\ref{7.7}), we have
\begin{align}
\||D^{(i)}u^{(i)}|\|_{L^{\infty}(B^{(i)}_{p_{i}}(\frac{3}{2}))} \leq C(k,N).
\end{align}
In other word, on $B^{(i)}_{p_{i}}(\frac{3}{2})$, $u^{(i)}$ is a Lipschitz function with Lipschitz constant $C=C(k,N)$.
Note that $u^{(i)}(p_{i})=0$, hence $\| u^{(i)}\|_{L^{\infty}(B^{(i)}_{p_{i}}(\frac{3}{2}))}\leq C(k,N)$.

By Theorem \ref{AA}, up to a subsequence, $u^{(i)}$ converge uniformly on any compact subset to some function $u^{(\infty)}:B^{(\infty)}_{p_{\infty}}(\frac{3}{2})\rightarrow\mathbb{R}$.
$u^{(\infty)}$ is Lipschitz on $B^{(\infty)}_{p_{\infty}}(\frac{3}{2})$, and $u^{(\infty)}(p_{\infty})=0$.

The following proposition is implied by Corollary 4.5 in \cite{AH17} or implied by the proof of Corollary 3.3 in \cite{ZZ17}.
\begin{prop}\label{2.7777}
Suppose the $\mathrm{RCD}^{*}(K,N)$ spaces $(X_{i}, p_{i}, d_{i}, m_{i})$ converge in the pointed measured Gromov-Hausdorff distance to $(X_{\infty}, p_{\infty}, d_{\infty}, m_{\infty})$.
Suppose $f_{i}$ is a harmonic function defined on $B^{(i)}_{p_{i}}(R)\subset X_{i}$ such that $\mathrm{Lip}f_{i}\leq L$ for every $i$ and some constant $L$, and $f_{\infty}$ is a function on $B^{(\infty)}_{p_{\infty}}(R)\subset X_{\infty}$.
If $f_{i}$ converge to $f_{\infty}$ uniformly on $B^{(\infty)}_{p_{\infty}}(R)$, then $f_{\infty}$ is harmonic on $B^{(\infty)}_{p_{\infty}}(R)$, and
\begin{align}
\lim_{i\rightarrow\infty}\int_{B^{(i)}_{p}(r)}|D^{(i)}f_{i}|^{2}dm_{i}
=\int_{B^{(\infty)}_{p_{\infty}}(r)}|D^{(\infty)}f_{\infty}|^{2}dm_{\infty},
\end{align}
holds for any $r\in(0,R)$.
\end{prop}

Since every $u^{(i)}$ is a harmonic function and (\ref{7.003}) holds, by Proposition \ref{2.7777}, $u^{(\infty)}$ is harmonic on $B^{(\infty)}_{p_{\infty}}(\frac{3}{2})$, and
\begin{align}
\int_{B^{(\infty)}_{p_{\infty}}(1)}|D^{(\infty)} u^{(\infty)}|^{2}dm_{\infty}=1.
\end{align}
In particular, $u^{(\infty)}$ is a non-constant harmonic function.

Since $(X,d,m)$ does not split, by Proposition \ref{main-1.4}, $(X_{\infty},d_{\infty},m_{\infty})$ is isomorphic to some $([0,\infty),d_{\mathrm{Eucl}},h\mathcal{L}^{1})$, where $h:[0,\infty)\rightarrow (0,\infty)$ is a locally Lipschitz function.
Then the existence of $u^{(\infty)}$ contradicts the following fact: for any $r>0$, there is no non-constant harmonic function on $([0,r),d_{\mathrm{Eucl}},h\mathcal{L}^{1})$.
This fact can be proved as follows.
Every harmonic function $f$ on $([0,r),d_{\mathrm{Eucl}},h\mathcal{L}^{1})$ must satisfy
\begin{align}\label{7.004}
\int f'\varphi'hd\mathcal{L}^{1}=0
\end{align}
for any Lipschitz function $\varphi$ with $\mathrm{supp}(\varphi)\subset\subset [0,r)$.
Define $g:[0,\infty)\rightarrow \mathbb{R}$ by $g(t)=\int_{1}^{t}f'hd\mathcal{L}^{1}$, then by (\ref{7.004}), we can derive that $g$ is a constant function, and thus $f'\equiv0$, i.e. $f$ is also a constant function.

The proof of Theorem \ref{main-1.5} is completed.
\end{proof}

\end{document}